\documentclass[a4paper,11pt]{article}

\usepackage{amsmath}
\usepackage{amsthm}
\usepackage{amssymb}

\usepackage[latin1]{inputenc}

\usepackage[dvips=true,%
a4paper=true,%
pdftitle={titre},%
pdfauthor={Jean-Francois Burnol},%
pdfstartview=FitH,%
pdfpagemode=None%
]{hyperref}

\def\ladate{July 27, 2010}

\newtheorem{theorem}{Theorem}

\newtheorem{corollary}[theorem]{Corollary}
\newtheorem{proposition}[theorem]{Proposition}

\newtheorem{definition}{Definition}
\theoremstyle{definition}
\newtheorem{remark}{Remark}

 \newcommand{\ZZ}{{\mathbb Z}}
 
 \newcommand{\RR}{{\mathbb R}}
 \newcommand{\CC}{{\mathbb C}}

\newcommand{\cA}{{\mathcal A}}
\newcommand{\cB}{{\mathcal B}}
\newcommand{\cC}{{\mathcal C}}

\newcommand{\cE}{{\mathcal E}}
\newcommand{\cF}{{\mathcal F}}

\newcommand{\cK}{{\mathcal K}}

\DeclareMathOperator{\diag}{diag}
\DeclareMathOperator{\sh}{sh}
\DeclareMathOperator{\ch}{ch}

\begin{document}

\title{Paley-Wiener spaces with vanishing conditions and Painlev\'e~VI
transcendents}

\author{Jean-Fran\c cois Burnol}

\date{\ladate}

\maketitle

\begin{center}
\begin{small}

\begin{abstract}   We modify the classical Paley-Wiener spaces $PW_x$ of
entire functions of finite exponential type at most $x>0$, which are square
integrable on the real line, via the additional condition of vanishing at
finitely many complex points $z_1, \dots, z_n$. We compute the reproducing
kernels and relate their variations with respect to $x$ to a Krein
differential system, whose coefficient (which we call the $\mu$-function) and
solutions have determinantal expressions. Arguments specific to the case where
the ``trivial zeros'' $z_1, \dots, z_n$ are in arithmetic progression on the
imaginary axis allow us to establish for expressions arising in the theory a
system of two non-linear first order differential equations. A computation,
having this non-linear system at his start, obtains
quasi-algebraic and among them rational Painlev\'e transcendents of the
sixth kind as certain quotients of such $\mu$-functions.
\end{abstract}

\bigskip

\begin{quote}
Universit\'e Lille 1\\ 
UFR de Math\'ematiques\\ 
Cit\'e scientifique M2\\ 
F-59655 Villeneuve d'Ascq\\ 
France\\
burnol@math.univ-lille1.fr\\
\end{quote}

\end{small}

\end{center}

\setlength{\normalbaselineskip}{16pt}
\baselineskip\normalbaselineskip
\setlength{\parskip}{6pt}

\section{Introduction and summary of results}

Let $\phi\in L^2(\RR,dt)$ and $\cF(\phi)(z) = \int_{-\infty}^\infty
\phi(t)e^{izt}\,dt$ its Fourier transform. When $\phi$ is supported in
$(-x,x)$, $f(z) = \cF(\phi)(z)$ is an entire function of exponential type at
most $x$. Conversely the Paley-Wiener theorem identifies the vector space
$PW_x$ of entire functions of exponential type at most $x$, square-integrable
on the real line, as the Hilbert space of such Fourier transforms.  Our
convention for our scalar products is for them to be conjugate linear in the
first factor and complex linear in the second factor. Specifically
$(\phi,\psi) = \int_\RR \overline{\phi(t)}\psi(t)\,dt$, hence for the
transforms $f$ and $g$: $(f,g) = \frac1{2\pi}\int_\RR \overline{f(z)}g(z) dz =
(\phi,\psi)$. 

The evaluator $Z_z$ is the element of $PW_x$ such
that 
\begin{equation}
\forall g\in PW_x,\; g=\cF(\psi),\qquad (Z_z,g) = g(z) = \int_{-x}^x e^{izt}\psi(t)\,dt =
(\cF(e^{-i\overline z t}), g)
\end{equation}
Hence:
\begin{equation}
  Z_z(w) = 2\frac{\sin((\overline z - w)x)}{\overline z-w} = \frac{e^{i\overline z x}e^{-iw
      x} - e^{-i\overline z x}e^{iw x}}{i(\overline z - w)}
\end{equation}
Let $E(w) = e^{-i x w}$ and $E^*(w) = \overline{E(\overline w)}$. The evaluators in $PW_x$ are given by
 \begin{equation}
  \label{eq:E}
 Z_z(w) = (Z_w,Z_z) = \frac{\overline{E(z)}E(w) - \overline{E^*(z)}E^*(w)}{i(\overline z-w)}  
\end{equation}
Let us also define:
\begin{subequations}
  \begin{align}
    A(w) &= \frac12(E(w) + E^*(w))\\
    B(w) &= \frac i2(E(w) - E^*(w))
  \end{align}
\end{subequations}
Then $E = A - iB$, $A = A^*$, $B = B^*$ and:
\begin{equation}
  \label{eq:AB}
  (Z_w,Z_z) = Z_z(w) = 2 \frac{\overline{B(z)}A(w) - \overline{A(z)}B(w)}{\overline
    z - w}
\end{equation}
For the Paley-Wiener spaces, $A(w) = \cos(xw)$ is even and $B(w) = \sin(xw)$
is odd.

Let us consider generally a Hilbert space $H$, whose vectors are entire
functions, and such that the evaluations at complex numbers are continuous
linear forms, hence correspond to specific vectors $Z_z$. 
Let $\sigma = (z_1, \dots, z_n)$ be a finite sequence of distinct complex
numbers. We let $H^\sigma$ be the closed subspace of $H$
of functions vanishing at the $z_i$'s. Let 
\begin{equation}
  \gamma(z) = \frac1{(z-z_1)\dots (z-z_n)}
\end{equation}
and define $H(\sigma) = \gamma(z) H^\sigma$:
\begin{equation}
  H(\sigma) = \{ F(z) = \gamma(z)f(z)\;|\; f\in H, f(z_1) = \dots = f(z_n) = 0\}
\end{equation} 
We introduced this notion in \cite{trivial}. We  say that $F(z) = \gamma(z) f(z)$ is the ``complete'' form of $f$, and
refer to $z_1$, \dots, $z_n$ as  the ``trivial zeros'' of $f$. We give
$H(\sigma)$ the Hilbert space structure which makes $f\mapsto F$ an isometry
with $H^\sigma$. Let us note that evaluations $F\mapsto F(z)$ are again
continuous linear forms on this new Hilbert space of entire functions: this is
immediate if $z\notin\sigma$ and follows from the Banach-Steinhaus theorem if
$z\in\sigma$. We thus define  $\cK_z$ in $H(\sigma)$ to be the evaluator at $z$:
\begin{equation}
  \cK_z\in H(\sigma)\quad \forall F\in H(\sigma) \quad F(z) = (\cK_z,F)_{H(\sigma)}
\end{equation}

Here is a summary of the results presented here. We start by showing how to find entire functions $\cA_\sigma$ and
$\cB_\sigma$, real on the real line,  such that:
\begin{equation}
  \label{eq:ABn} (\cK_w,\cK_z) = \cK_z(w) = 2
\frac{\overline{\cB_\sigma(z)}\cA_\sigma(w) -
\overline{\cA_\sigma(z)}\cB_\sigma(w)}{\overline z - w}
\end{equation} This will be done under the following hypotheses: (1) the
initial Hilbert space of entire functions $H$ satisfies the axioms of
\cite{Bra}, hence its evaluators $Z_z$  are given by a formula \eqref{eq:AB}
for some entire functions $A$ and $B$ which are real on the real line, (2) $A$
can be chosen even and $B$ can be chosen odd, and (3) the added ``trivial
zeros'' are purely imaginary. The produced functions $\cA_\sigma$ and
$\cB_\sigma$ giving the reproducing kernel \eqref{eq:ABn} of the modified
space $H(\sigma)$ will be respectively even and odd. The  restrictive
hypotheses (2) and (3) can be disposed of, as is explained in companion paper
\cite{trivial}. We follow here another method, which proves  formulas of a
different type than those available from the general treatment
\cite{trivial}. The interested reader will find in \cite{trivial} the easy
arguments establishing that $H(\sigma)$ verifies the axioms of \cite{Bra} if
the initial space $H$ does: this explains a priori why indeed a formula of the
type \eqref{eq:ABn} has to exist if \eqref{eq:AB} holds for $H$.

Then, we examine the case of a dependency of the initial space on a parameter
$x$. Assuming that the initial $A_x$ and $B_x$ obey a first order
differential system of the Krein type \cite{KreinScattering,KreinContinual} as functions of $x$
(involving as coefficient what we call a $\mu$-function) we prove that the new
$\cA_{\sigma,x}$ and $\cB_{\sigma,x}$ do as well (in other words we compute
the $\mu_\sigma$-function in terms of the initial $\mu$-function). The result
is already notable when we start from the classical Paley-Wiener spaces for
which the initial $\mu(x)$-function ($x>0$) vanishes identically. It will be
achieved through establishing a ``pre-Crum formula'' for the effect of  
Darboux transformations on Schr\"odinger equations linked into Krein systems.

The  final part of the paper  establishes the main result.  We consider the
classical Paley-Wiener spaces $PW_x$ modified by imaginary trivial zeros in an
arithmetic progression $\sigma$. We prove that certain quotients of the
$\mu$-functions associated to the spaces $PW_x(\sigma)$ obey the 
Painlev\'e~VI differential equation.

\section{A  determinantal identity}

The following identity is quasi identical with a formula of Okada
\cite[Theorem 4.2]{Okada} and immediately equivalent to it. We give a different
proof.

\begin{theorem}\label{thm:repdet}
Let there be given indeterminates $u_i$, $v_i$, $k_i$, $x_i$, $y_i$, $l_i$, for
$1\leq i \leq n$. We define the following $n\times n$ matrices
\begin{equation}
  U_n =
  \begin{pmatrix}
    u_1 & u_2 & \dots & u_n\\
k_1 v_1 & k_2 v_2 & \dots & k_n v_n\\
k_1^2 u_1 & k_2^2 u_2 & \dots & k_n^2 u_n\\
\vdots & \dots & \dots & \vdots
  \end{pmatrix}
\qquad
V_n =
  \begin{pmatrix}
    v_1 & v_2 & \dots & v_n\\
k_1 u_1 & k_2 u_2 & \dots & k_n u_n\\
k_1^2 v_1 & k_2^2 v_2 & \dots & k_n^2 v_n\\
\vdots & \dots & \dots & \vdots
  \end{pmatrix}
\end{equation}
where the rows contain alternatively $u$'s and $v$'s. Similarly:
\begin{equation}
  X_n =
  \begin{pmatrix}
    x_1 & x_2 & \dots & x_n\\
l_1 y_1 & l_2 y_2 & \dots & l_n y_n\\
l_1^2 x_1 & l_2^2 x_2 & \dots & l_n^2 x_n\\
\vdots & \dots & \dots & \vdots
  \end{pmatrix}
\qquad
Y_n =
  \begin{pmatrix}
    y_1 & y_2 & \dots & y_n\\
l_1 x_1 & l_2 x_2 & \dots & l_n x_n\\
l_1^2 y_1 & l_2^2 y_2 & \dots & l_n^2 y_n\\
\vdots & \dots & \dots & \vdots
  \end{pmatrix}
\end{equation}
There holds
  \begin{equation}
    \det_{1\leq i,j\leq n}(\frac{u_iy_j-v_ix_j}{l_j-k_i}) =
\frac1{\prod_{i,j} (l_j - k_i)}\;
\begin{vmatrix}
  U_n\mathstrut &X_n\\
V_n\mathstrut&Y_n
\end{vmatrix}_{2n\times 2n}
\end{equation}
\end{theorem}

\begin{proof}
  Let $A$, $B$, $C$, $D$ be $n\times n$ matrices, with  $A$ and $C$ invertible. Using $
  \left(\begin{smallmatrix} A&B\\C&D
    \end{smallmatrix}\right) = \left(\begin{smallmatrix}
      A&0\\0&C
    \end{smallmatrix}\right)\left(\begin{smallmatrix}
      I&A^{-1}B\\I&C^{-1}D
    \end{smallmatrix}\right)$ we obtain
  \begin{equation}
    \begin{vmatrix}
      A& B\\C&D
    \end{vmatrix} = |A||C| |C^{-1}D - A^{-1} B|
  \end{equation}
  where vertical bars denote determinants. Let $d(u) = \textrm{diag}(u_1,
  \dots, u_n)$ and $p_u = \prod_{1\leq i\leq n} u_i$. We define
  similarly $d(v)$, $d(x)$, $d(y)$ and $p_v, p_x, p_y$. From the previous
  identity we get
  \begin{equation}
    \begin{split}
      \begin{vmatrix}
        A d(u)& Bd(x)\\Cd(v) &D d(y)
      \end{vmatrix} &= |A||C|\; p_up_v\; \Big|d(v)^{-1}C^{-1}Dd(y)-
      d(u)^{-1}A^{-1} Bd(x)\Big| \\ &= |A||C|\Big|d(u)C^{-1}Dd(y)-
      d(v)A^{-1}Bd(x)\Big|
    \end{split}
  \end{equation}
  The special case $A=C$, $B=D$, gives
  \begin{equation}\label{eq:det}
    \begin{vmatrix}
      A d(u)& Bd(x)\\Ad(v) &B d(y)
    \end{vmatrix}_{2n\times 2n} = \det(A)^2  \det_{1\leq i,j\leq n}((u_iy_j-v_ix_j)(A^{-1}B)_{ij})
  \end{equation}
Let  $W(k)$ be the Vandermonde matrix with rows  $(1\dots 1)$,  $(k_1\dots
k_n)$, $(k_1^2\dots
k_n^2)$, \dots, and $\Delta(k) = \det W(k)$ its determinant. Let 
\begin{equation}
K(t) = \prod_{1\leq m \leq
  n} (t-k_m)\, 
\end{equation}
and  let $\cC$ be the $n \times n$ matrix $(c_{im})_{1\leq
  i,m\leq n}$, where the $c_{im}$'s are defined by the partial fraction expansions:
  \begin{equation}
    1\leq i \leq n\qquad \frac{t^{i-1}}{K(t)} = \sum_{1\leq m\leq n} \frac{c_{im}}{t-k_m}
  \end{equation}
We have the two matrix equations:
\begin{subequations}
  \begin{align}
    \cC &= W(k)\diag(K'(k_1)^{-1}, \dots, K'(k_n)^{-1})\\
    \cC\cdot \big(\frac1{l_j-k_m}\big)_{1\leq m, j\leq n} &=
    W(l)\diag(K(l_1)^{-1}, \dots, K(l_n)^{-1})
  \end{align}
\end{subequations}
This gives the (well-known) identity:
  \begin{equation}
    \left(\frac1{l_j-k_m}\right)_{1\leq m,j\leq
      n} = \diag(K'(k_1),\dots, K'(k_n))W(k)^{-1}W(l)\diag(K(l_1)^{-1}, \dots, K(l_n)^{-1})
  \end{equation}
We can thus rewrite the determinant we want to compute as:
  \begin{equation}
    \left| \frac{u_iy_j-v_ix_j}{l_j-k_i}\right|_{1\leq i,j\leq n} = \prod_m
    K'(k_m)\prod_j K(l_j)^{-1}  \Big|
    (u_iy_j-v_ix_j)(W(k)^{-1}W(l))_{ij}\Big|_{n\times n}
  \end{equation}
We shall now make use of \eqref{eq:det} with $A = W(k)$ and $B =
  W(l)$.
  \begin{equation}
    \begin{split}
      \left| \frac{u_iy_j-v_ix_j}{l_j-k_i}\right|_{1\leq i,j\leq n} &=
      \Delta(k)^{-2} \prod_m K'(k_m)\prod_j K(l_j)^{-1} \begin{vmatrix} W(k)
        d(u)& W(l)d(x)\\W(k) d(v) &W(l) d(y)
      \end{vmatrix}\\ &= \frac{(-1)^{\frac{n(n-1)}2}}{\prod_{i,j} (l_j -
        k_i)}\; \begin{vmatrix} W(k) d(u)& W(l)d(x)\\W(k) d(v) &W(l) d(y)
      \end{vmatrix}_{2n\times 2n}
    \end{split}
  \end{equation}
The sign $(-1)^{n(n-1)/2} = (-1) ^{[\frac n2]}$ is the signature of the
permutation which exchanges rows $i$ and $n+i$ for $i = 2, 4 , \dots, 2[\frac
n2]$ and transforms the determinant on the right-hand side into 
 $    \begin{vmatrix}
      U_n\mathstrut &X_n\\
      V_n\mathstrut&Y_n
    \end{vmatrix}$. This concludes the proof.
\end{proof}

\section{$A$ and $B$  for spaces with imaginary trivial zeros}

Just using the existence of continuous evaluators but not yet \eqref{eq:AB},
we have by a simple argument of orthogonal projection (see
\cite{trivial}):
\begin{proposition}\label{thm:gram}
Let $H$ be a Hilbert space of entire functions with continuous evaluators
$Z_z$: $\forall f\in H\; f(z) = (Z_z,f)$. Let
$\sigma = (z_1,\dots,z_n)$ be a finite sequence of distinct complex numbers
with associated evaluators $Z_1$, \dots, $Z_n$, assumed to be linearly
independent. Let $H(\sigma)$ be the Hilbert space of entire functions which are complete
forms of the elements of $H$ vanishing on $\sigma$. The evaluators of
$H(\sigma)$ are given by:
  \begin{equation}\label{eq:K}
    \cK_z(w) = \frac{\gamma(w)\overline{\gamma(z)}}{G_n}
    \begin{vmatrix}
      (Z_1,Z_1)&\dots &(Z_1,Z_n) &(Z_1,Z_z)\\
      (Z_2,Z_1)&\dots &(Z_2,Z_n) &(Z_2,Z_z)\\
      \vdots&\dots&\vdots&\vdots\\
      (Z_w,Z_1)&\dots& (Z_w,Z_n)&(Z_w,Z_z)
    \end{vmatrix}
  \end{equation}
where $G_n>0$ is the principal $n\times n$ minor of the matrix.
\end{proposition}

Recalling the form \eqref{eq:AB} of the reproducing kernel:
\begin{equation}
  (Z_{w_1},Z_{w_2}) = 2 \frac{\overline{B(w_2)}A(w_1) - \overline{A(w_2)}B(w_1)}{\overline
    w_2 - w_1}
\end{equation}
we see that  the choices:
\begin{subequations}
  \begin{align}
  1\leq i\leq n:&\quad u_i = A(z_i), v_i = B(z_i), k_i = z_i   \\
&\quad u_{n+1} = A(w), v_{n+1} = B(w), k_{n+1} = w,\\
1\leq j \leq n:&\quad x_j = \overline{A(z_j)}, y_j = \overline{B(z_j)},
  l_j = \overline{z_j}, \\
&\quad x_{n+1} = \overline{A(z)}, y_{n+1} = \overline{B(z)}, l_{n+1} = \overline z
  \end{align}
\end{subequations}
allow to make use of Theorem \ref{thm:repdet}. This gives:
\begin{equation}
  \cK_z(w) = \frac{2^{n+1}\gamma(w)\overline{\gamma(z)}(-1)^n\gamma^*(w)\gamma(\overline z)}{G_n\cdot \prod_{1\leq
      i,j \leq n} (\overline{z_j} - z_i)\cdot(\overline z -w)}   
\begin{vmatrix}
  U_{n,w}\mathstrut &X_{n,\overline z}\\
V_{n,w}\mathstrut&Y_{n,\overline z}
\end{vmatrix}_{(2n+2)\times (2n+2)} 
\end{equation}
with
\begin{equation}
  U_{n,w} =
  \begin{pmatrix}
    A(z_1) & A(z_2) & \dots & A(z_n) &A(w)\\
z_1 B(z_1)& z_2 B(z_2) & \dots & z_n B(z_n) & w B(w)\\
z_1^2 A(z_1) & z_2^2 A(z_2) & \dots & z_n^2 A(z_n) & w^2 A(w)\\
\vdots & \dots & \dots & \vdots& \vdots
  \end{pmatrix}
\end{equation}
\begin{equation}
V_{n,w} =
  \begin{pmatrix}
    B(z_1) & B(z_2) & \dots & B(z_n) &B(w)\\
z_1 A(z_1) & z_2 A(z_2) & \dots & z_n A(z_n) & w A(w)\\
z_1^2 B(z_1) & z_2^2 B(z_2) & \dots & z_n^2 B(z_n)& w^2 B(w)\\
\vdots & \dots & \dots & \vdots& \vdots
  \end{pmatrix}
\end{equation}
\begin{equation}
  X_{n,\overline z} = \overline{U_{n,z}}
\end{equation}
\begin{equation}
Y_{n,\overline z} = \overline{V_{n,z}}
\end{equation}

We shall now make the following hypotheses:
$(1)$ the $z_i$'s are purely imaginery, $(2)$
$A$ is even and $B$ is odd.
Then $\overline{A(z_i)} = A(\overline z_i) = A(-z_i) = A(z_i)$, and  $\overline{B(z_i)} = B(\overline{z_i}) =
B(-z_i) = -B(z_i)$. The first $n$ columns of the matrix $U_{n,w}$ are thus
real and identical with the first $n$ columns of $X_{n,\overline z}$. The first $n$
columns of the matrix $V_{n,w}$ are purely imaginery and thus the opposite of
the first $n$ columns of $Y_{n,\overline z}$.

In order to compute the determinant $\begin{vmatrix}
  U_{n,w}\mathstrut &X_{n,\overline z}\\
V_{n,w}\mathstrut&Y_{n,\overline z}
\end{vmatrix}_{(2n+2)\times (2n+2)} $, we  substract, for
$1\leq j\leq n$, column $j$ to column $j+n+1$. This sets to zero all columns of $X_{n,\overline
  z}$ except its last
 and multiplies by $2$ the $n$ first columns of $Y_{n,\overline z}$. We then
apply a Laplace expansion using the $(n+1)\times(n+1)$ minors built with the first and last $(n+1)$
rows. If the top minor has
both the $w$ and $\overline z$ columns,  its complementary bottom minor 
will have two proportional columns hence vanish. There are thus only two
contributions, and taking (various) signs into account we obtain:
\begin{equation}
  \begin{vmatrix}
  U_{n,w}\mathstrut &X_{n,\overline z}\\
V_{n,w}\mathstrut&Y_{n,\overline z}
\end{vmatrix}_{(2n+2)\times (2n+2)} = 2^n \det(U_{n,w})\det(Y_{n,\overline z})-
 \det(X_{n,\overline z})(-2)^n \det(V_{n,w})
\end{equation}
So:
\begin{subequations}
  \begin{align}
  \cK_z(w) &= \frac{2^{2n+1}\gamma(w)\overline{\gamma(z)}}{G_n} 
  \frac{\gamma(\overline z)(-1)^n \gamma^*(w)}{\prod_{i,j\leq n} (\overline z_j - z_i)}
  \cdot \frac{\det(U_{n,w})\overline{\det(V_{n,z})}
-\overline{\det(U_{n,z})}(-1)^n\det(V_{n,w})}{\overline z - w}
  \end{align}
\end{subequations}
Let us also compute the Gram determinant $G_n$. The determinantal
identity gives:
\begin{equation}\label{eq:Gn}
  G_n = \frac{2^n}{\prod_{i,j} (\overline z_j - z_i)}\;
\begin{vmatrix}
  U_n\mathstrut &U_n\\
V_n\mathstrut&-V_n
\end{vmatrix}_{2n\times 2n} =  \frac{2^{2n}(-1)^n}{\prod_{1\leq i,j\leq n} (\overline z_j - z_i)}
\det(U_n)\det(V_n)
\end{equation}
Finally:
\begin{equation}
  \cK_z(w)  = 2\gamma(w)\gamma^*(w)\overline{\gamma(z)\gamma^*(z)}\frac{\det(U_{n,w})\overline{\det(V_{n,z})}
-\overline{\det(U_{n,z})}(-1)^n\det(V_{n,w})}{\det(U_n)\det(V_n)(\overline z - w)}
\end{equation}
Taking into account that $i^n \det(V_n)$ is real we get:
\begin{subequations}
  \begin{align}
  \cK_z(w)  &= 2\gamma(w)\gamma^*(w)\overline{\gamma(z)\gamma^*(z)}\frac{\det(U_{n,w})\overline{i^n\det(V_{n,z})}
-\overline{\det(U_{n,z})}i^n\det(V_{n,w})}{\det(U_n)(-i)^n\det(V_n)(\overline z -
w)}\\
&= 2\frac{\gamma(w)\gamma^*(w)\overline{\gamma(z)\gamma^*(z)}}{\overline z
  -w}\left(\overline{(-1)^n
    \frac{\det(V_{n,z})}{\det V_n}} \frac{\det(U_{n,w})}{\det U_n} -
  \overline{\frac{\det(U_{n,z})}{\det U_n}}\frac{(-1)^n \det(V_{n,w})}{\det V_n}\right)
  \end{align}
\end{subequations}

The following has been obtained:

\begin{theorem} \label{thm:3}
Let $H$ be a Hilbert space of entire functions with
reproducing kernel  $Z_z(w) = 2 \frac{\overline{B(z)}A(w) -
\overline{A(z)}B(w)}{\overline z -w}$, where the entire functions $A$ and $B$ are
real on the real line and respectively even and odd. Let $\sigma =
(z_i)_{1\leq i\leq n}$ be a finite sequence of distinct purely imaginary
numbers. We assume that the associated evaluators are linearly independent,
and also that $z_i+z_j\neq 0$ for all $i,j$. Let $H(\sigma)$ be the Hilbert
space of the functions $\gamma(z) f(z)$, where $\gamma(z) = \prod_i
\frac1{z-z_i}$ and $f$ is in $H$ with $f(z) = 0$ for $z\in\sigma$. Let
  \begin{subequations}
  \begin{equation}\label{eq:aa}
    \cA_\sigma(w) = \frac{(-1)^{\frac{n(n-1)}2}\gamma(w)\gamma^*(w)}{\det U_n}    \begin{vmatrix}
      A(z_1) & A(z_2) & \dots & A(z_n) &A(w)\\
      z_1 B(z_1) & z_2 B(z_2) & \dots & z_n B(z_n) & w B(w)\\
      z_1^2 A(z_1) & z_2^2 A(z_2) & \dots & z_n^2 A(z_n) & w^2 A(w)\\
      \vdots & \dots & \dots & \dots& \vdots
    \end{vmatrix}_{(n+1)\times(n+1)}
\end{equation}
\begin{equation}
\label{eq:bb}    \cB_\sigma(w) = \frac{(-1)^{\frac{n(n+1)}2}\gamma(w)\gamma^*(w)}{\det V_n}
    \begin{vmatrix}
      B(z_1) & B(z_2) & \dots & B(z_n) &B(w)\\
      z_1 A(z_1) & z_2 A(z_2) & \dots & z_n A(z_n) & w A(w)\\
      z_1^2 B(z_1) & z_2^2 B(z_2) & \dots & z_n^2 B(z_n) & w^2 B(w)\\
      \vdots & \dots & \dots & \dots& \vdots
    \end{vmatrix}_{(n+1)\times(n+1)}
    \end{equation}
\end{subequations}
where  the denominators $\det U_n$ and $\det V_n$ are the principal $n\times n$
  minors of the numerators. The space $H(\sigma)$ has evaluators $\cK_z$
  satisfying the formula:
  \begin{equation}\label{eq:Kn}
    \cK_z(w) = (\cK_w,\cK_z) = 2 \frac{\overline{\cB_\sigma(z)}\cA_\sigma(w) - \overline{\cA_\sigma(z)}\cB_\sigma(w)}{\overline z -w}
  \end{equation}
  The functions $\cA_\sigma$ and $\cB_\sigma$ are entire, real on the real line, $\cA_\sigma$ is even and
  $\cB_\sigma$ is odd. 
\end{theorem}

\begin{remark}
  The additional $(-1)^{\frac{n(n-1)}2}$ is to
  make $\cA_\sigma(it)>0$ and $-i\cB_\sigma(it)>0$ for $t>0$, at least in the
  case of the Paley-Wiener spaces (this sign is easily determined from the
  asymptotics as $t\to +\infty$; let us also mention that
  $-i\cB_\sigma(it)\cA_\sigma(it)>0$ for $t>0$, from \eqref{eq:Kn} and if $H(\sigma)\neq\{0\}$). We
  observe that if the initial $A$ and $B$ verify the normalization
  $\frac{-iB(it)}{A(it)}\to_{t\to+\infty} 1$ then the new $\cA_\sigma$ and
  $\cB_\sigma$ also. This normalization has proven to be more natural in this
  and other investigations, than other normalizations such as, for example,
  $A(0) = 1$ (when possible).
\end{remark}

\begin{remark} A formula of the type \eqref{eq:Kn} for evaluators in a Hilbert
space of entire functions is guaranteed by the axiomatic framework of
\cite{Bra}. The passage from an $H$ to an $H(\sigma)$ is compatible to these
axioms (cf. \cite{trivial}), hence existence of an  $\cE_\sigma =
\cA_\sigma - i\cB_\sigma$ function was known in advance. Determination of a
suitable $\cE_\sigma$, without any of the restrictive hypotheses made here, is
achieved in \cite{trivial} with another method. Confrontation of the results
thus establishes some interesting identities.
\end{remark}

Let us record a special case of the computation \eqref{eq:Gn} of the Gram
determinant $G_n$, using the notation $W(f_1, \dots, f_n)$ for Wronskian
determinants $\det(f_j^{(i)})_{1\leq i, j \leq n}$ (derivatives with
respect to $x$):

\begin{proposition}\label{prop}
  The following identity holds:
  \begin{equation} 
\left| \frac{\sh((\kappa_i+\kappa_j)x)}{\kappa_i+\kappa_j}
\right|_{1\leq i, j\leq n}  = \frac{W(\ch(\kappa_1 x), \dots, \ch(\kappa_nx))\cdot W(\sh(\kappa_1 x), \dots,
    \sh(\kappa_nx))}{\prod_{1\leq i \leq n} \kappa_i \; \prod_{1\leq i<j\leq n}
  (\kappa_i+\kappa_j)^2} 
  \end{equation}
\end{proposition}
\begin{proof}
Let $u_i = \ch(\kappa_i x)$, $v_i= \sh(\kappa_i x)$, $k_i = \kappa_i$, $x_j =
\ch(\kappa_j x) = u_j$, $y_j = -\sh(\kappa_j x) = -v_j$, $l_j = -\kappa_j =
-k_j$. With these choices:
  \begin{equation}
      \frac{\sh((\kappa_i+\kappa_j)x)}{\kappa_i+\kappa_j} = \frac{u_iy_j-v_ix_j}{l_j-k_i}
  \end{equation}
By Theorem \ref{thm:repdet}:
\begin{equation}
    \det_{1\leq i,j\leq n}(\frac{u_iy_j-v_ix_j}{l_j-k_i}) = \frac1{\prod_{i,j}
      (l_j - k_i)}\;
    \begin{vmatrix}
      U_n\mathstrut &X_n\\
      V_n\mathstrut&Y_n
    \end{vmatrix}_{2n\times 2n} = \frac{(-1)^n(-2)^n}{\prod_{i,j}
      (\kappa_i + \kappa_j)} \det(U_n)\det(V_n)
\end{equation}
where we used $U_n = X_n$, $V_n = -Y_n$. As $\det U_n = W(\ch(\kappa_1 x),
\dots, \ch(\kappa_nx))$ and $\det V_n
=W(\sh(\kappa_1 x), \dots, \sh(\kappa_nx))$,
this completes the proof.
\end{proof}

\section{Crum formulas for Darboux transformations of  Krein systems}

All derivatives in this chapter will be with respect to a variable $x$. We
are interested in  differential systems of
the Krein type:
\begin{equation}\label{eq:S}
(S) \left\{\quad  \begin{matrix}
    a' - \mu a &= - k b\\
    b' + \mu b &= +k a
  \end{matrix}\right.
\end{equation}
Krein uses systems of this type in particular in his approach \cite{KreinScattering} to Inverse
Scattering Theory and in his continuous analogues to topics of Orthogonal
Polynomial Theory \cite{KreinContinual}. The system couples two Schr\"odinger equations:
\begin{subequations}
  \begin{align}\label{eq:schrod+}
  -a'' + V^+ a &= k^2 a\qquad\text{with }V^+ = \mu^2 + \mu'\\ \label{eq:schrod-}
-b'' + V^- b &= k^2 b\qquad\text{with }V^- = \mu^2 - \mu'
  \end{align}
\end{subequations}
It proves quite convenient to introduce the notion of a tau-function, which is
a function such that:
\begin{equation}
  \label{eq:deftau}
  \mu^2 = - (\log \tau)''
\end{equation}
We shall also use the notation $\lambda = (\log \tau)'$, so that $\mu^2 = - \lambda' $.

The well-known Darboux transformation \cite[\textsection6]{Darboux} transforms
the solutions of a Schrödinger equation $-f'' + V f = E f$ into solutions of
another one, and the formulas of Crum \cite{Crum} give Wronskian expressions
for both solutions and potentials after successive such Darboux
transformations. In this chapter we introduce a notion of ``simultaneous'' or
``linked''  such transformations which act  at the level of the Krein
system \eqref{eq:S}. This provides a kind of refinement to the formula of
Crum, the change of the two potentials being lifted  to the change of the
``tau'' function. We did not find in the litterature the results we prove
here, but it is so extensive that we may have missed some important
contributions.

We make use also of couples $(\alpha,\beta)$ of the type $(a,-ib)$. Hence we also consider the
differential systems:
\begin{equation}\label{eq:T}
 (T) \left\{\quad \begin{matrix}
    \alpha' - \mu \alpha &= +\kappa \beta\\
    \beta' + \mu \beta &= +\kappa \alpha
  \end{matrix}\right.
\end{equation}
It corresponds to $(S)$ \eqref{eq:S} via $k=i\kappa$, $a = \alpha$,  $b = i \beta$. The
Schr\"odinger equations become:
\begin{subequations}
  \begin{align}
\alpha''  &= (\kappa^2+ V^+) \alpha\\
\beta'' &= (\kappa^2 +V^-)\beta.
  \end{align}
\end{subequations}

\begin{theorem}\label{thm:simdarboux} Let $\kappa\in\CC$ and let $(\alpha,\beta)$ be a solution of the differential
  system $(T)$ \eqref{eq:T}, with neither $\alpha$ nor $\beta$ identically zero. The
  simultaneous Darboux
  transformations: 
  \begin{subequations}
  \begin{align}
    a\to a_1 &= a' - \frac{\alpha'}\alpha a\label{eq:Da}\\
   b\to b_1&= b' - \frac{\beta'}\beta b\label{eq:Db}
    \end{align}
\end{subequations}
 transform any solution $(a,b,k)$ of the differential system $(S)$ \eqref{eq:S} into a
 solution $(a_1,b_1,k)$ of a transformed system:
\begin{subequations}
  \begin{align}
  (S_1) \left\{\quad\begin{matrix}
    a_1' - \mu_1 a_1 &= -k b_1\\
    b_1' + \mu_1 b_1 &= +k a_1
  \end{matrix}\right.
  \end{align}
\end{subequations}
where  the new coefficient $\mu_1$ is 
\begin{equation}
  \mu_1 =  \mu - \frac{d}{dx} \log\frac\alpha\beta
\end{equation}
If $\mu^2 = -(\log \tau)''$ then $\mu_1^2 = -(\log \tau_1)''$ with
\begin{equation}
    \tau_1 = \tau \alpha \beta
\end{equation}
\end{theorem}

\begin{proof}
From  $\alpha a_1 = \left|\begin{smallmatrix} \alpha&a\\ \alpha'&a'
    \end{smallmatrix}\right|$, we get $(\alpha a_1)' =
  \left|\begin{smallmatrix} \alpha&a\\ \alpha''&a''
    \end{smallmatrix}\right| = \left|\begin{smallmatrix}
      \alpha&a\\ (V^+ + \kappa^2)\alpha&(V^+ - k^2)a
    \end{smallmatrix}\right| = -(k^2+\kappa^2)a\alpha$, which we rewrite as
  \begin{equation}
    a_1' + \frac{\alpha'}\alpha a_1 =  -(k^2+\kappa^2)a = -k(b' + \mu b) - \kappa^2 a
  \end{equation}
Further  \begin{equation}
    \alpha a_1 = \alpha(-kb+\mu a)- (\kappa \beta +\mu\alpha)a = -k \alpha b -
    \kappa \beta a
  \end{equation}
 Eliminating $a$ gives:
  \begin{equation}
    a_1' + \frac{\alpha'}\alpha a_1 - \kappa\frac\alpha\beta a_1 =  -k(b' + \mu b)+ k
    \kappa\;\frac\alpha\beta b
  \end{equation}
Using  $\kappa \frac\alpha\beta = \frac{\beta'}\beta + \mu$:
  \begin{equation}
    a_1' + (\frac{\alpha'}\alpha  - \frac{\beta'}\beta - \mu)a_1 = -k(b' -
    \frac{\beta'}\beta b)
  \end{equation}
With the definitions $\mu_1 = \mu - \frac{\alpha'}\alpha  +
\frac{\beta'}\beta$ and $b_1 = b' - \frac{\beta'}\beta b$ this gives indeed:
\begin{equation}
   a_1' - \mu_1 a_1 = -k b_1
\end{equation}
From $\beta b_1 =
  \left|\begin{smallmatrix} \beta&b\\ \beta'&b'
    \end{smallmatrix}\right|$, we get $(\beta b_1)' =
  \left|\begin{smallmatrix} \beta&b\\ \beta''&b''
    \end{smallmatrix}\right| = \left|\begin{smallmatrix}
      \beta&b\\ (V^- + \kappa^2)\beta&(V^- - k^2)b
    \end{smallmatrix}\right| = -(k^2+\kappa^2)b\beta$, which gives:
  \begin{equation}
    b_1' + \frac{\beta'}\beta b_1 =  -(k^2+\kappa^2)b = k(a' - \mu a) - \kappa^2 b
  \end{equation}
On the other hand
  \begin{equation}
    \beta b_1 = \beta(k a-\mu b)- (\kappa \alpha - \mu\beta)b = k \beta a -
    \kappa \alpha b
  \end{equation}
Eliminating $b$ gives:
  \begin{equation}
    b_1' + \frac{\beta'}\beta b_1 - \kappa\frac\beta\alpha b_1 =  k(a' - \mu a)- k
    \kappa\;\frac\beta\alpha a
  \end{equation}
Using $\kappa \frac\beta\alpha = \frac{\alpha'}\alpha - \mu$ finally leads to:
  \begin{equation}
    b_1' + (\frac{\beta'}\beta  - \frac{\alpha'}\alpha + \mu)b_1 = k(a' -
    \frac{\alpha'}\alpha a) \implies b_1' + \mu_1 b_1 = + ka_1
  \end{equation}

Let $\lambda_1 = \lambda + \frac{\alpha'}\alpha +\frac{\beta'}\beta = (\log
\tau \alpha\beta)'$. We must also verify $\mu_1^2 = -\lambda_1'$.
  \begin{equation}
    \lambda_1' = \lambda' + (V^++\kappa^2) - (\frac{\alpha'}\alpha)^2 +(V^- +\kappa^2)  -(\frac{\beta'}\beta)^2= \mu^2 + 2\kappa^2   - (\frac{\alpha'}\alpha)^2 -(
    \frac{\beta'}\beta)^2 
  \end{equation}
  \begin{equation}
    \kappa^2 = \frac{(\alpha'  -\mu\alpha)(\beta'+\mu\beta)}{\alpha\beta} =
    \frac{\alpha'}\alpha\frac{\beta'}\beta - \mu\frac{\beta'}\beta +
    \mu\frac{\alpha'}\alpha - \mu^2 
  \end{equation}
  \begin{equation}
    \lambda_1' = -\mu^2 + 2  \frac{\alpha'}\alpha\frac{\beta'}\beta- 2 \mu(\frac{\beta'}\beta -
    \frac{\alpha'}\alpha)  -(\frac{\alpha'}\alpha)^2 -(
    \frac{\beta'}\beta)^2  = -\mu_1^2
  \end{equation}
\end{proof}

\begin{remark}\label{rem:T} A solution $(a,b,k)$ of system $(S)$ \eqref{eq:S}
corresponds to a solution $(\alpha, \beta,\kappa) = (a,-ib,-ik)$ of system
$(T)$ \eqref{eq:T}, and to a solution  $(\alpha, \beta,\kappa)$ of $(T)$ we
can switch to the solution $(\alpha,i\beta,i\kappa)$ of  $(S)$, having the
same logarithmic derivatives with respect to $x$. Hence it is just a matter of
arbitrary choice to consider the Darboux transformations to be associated to a
specific solution of $(T)$ rather than to a specific solution of
$(S)$. Moreover, the same Darboux transformations \eqref{eq:Da}, \eqref{eq:Db}
which are  associated to a given $(\alpha,\beta,\kappa)$ but now applied to a
triple $(\gamma,\delta, \xi)$, solution of $(T)$, produces a solution of the
transformed system
\begin{subequations}
  \begin{align}
  (T_1) \left\{\quad\begin{matrix}
    \gamma_1' - \mu_1 \gamma_1 &= \xi \delta_1\\
    \delta_1' + \mu_1 \delta_1 &= \xi \gamma_1
  \end{matrix}\right.
  \end{align}
\end{subequations}
of type $(T)$ associated to the new coefficient $\mu_1$.
\end{remark}

\begin{theorem}\label{thm:crumsystem}
Let there be given $n$ triples $(\alpha_j,\beta_j,\kappa_j)$, solutions of
the differential system $(T)$ \eqref{eq:T}. We assume that $\alpha_1$, \dots,
$\alpha_n$ are linearly independent, and $\beta_1$, \dots,
$\beta_n$ also. To each solution $(a,b,k)$ of the system 
\begin{equation}
(S) \left\{\quad  \begin{matrix}
    a' - \mu a &= - k b\\
    b' + \mu b &= +k a
  \end{matrix}\right.
\end{equation}
we associate
\begin{subequations}
  \begin{align}\label{eq:crum1}
      a_n &= \frac{W(\alpha_1, \dots, \alpha_n, a)}{W(\alpha_1, \dots,
      \alpha_n)}\\ \label{eq:crum2}
b_n &=  \frac{W(\beta_1, \dots, \beta_n, b)}{W(\beta_1, \dots,
      \beta_n)}
  \end{align}
\end{subequations}
Going from $(a,b)$ to $(a_n,b_n)$ is the result of the $n$ successive
simultaneous Darboux transformations \eqref{eq:Da} and \eqref{eq:Db}
associated to $(\alpha_1,\beta_1)$, \dots, $(\alpha_n,\beta_n)$ (themselves
transformed along the way). There holds:
\begin{equation}
(S_n) \left\{\quad  \begin{matrix}
    a_n' - \mu_n\; a_n &= - k b_n\\
    b_n' + \mu_n\; b_n &= +k a_n
  \end{matrix}\right.
\end{equation}
where the coefficient $\mu_n$ is given by:
  \begin{equation}\label{eq:muN}
\mu_n = \mu - \frac{d}{dx}\log \frac{W(\alpha_1, \dots,
      \alpha_n)}{W(\beta_1, \dots,
      \beta_n)}
\end{equation}
If furthermore one chooses a tau-function such that $\mu^2 = - (\log \tau)''$
then 
\begin{equation}
\mu_n^2 = - \frac{d^2}{dx^2}\log \tau_n
\end{equation}
where
\begin{equation}\label{eq:tauN}
\tau_n = \tau\cdot W(\alpha_1, \dots,\alpha_n)W(\beta_1, \dots,
      \beta_n)
\end{equation}
\end{theorem}

\begin{proof}  Let us consider first the simultaneous Darboux transformations
of system $(S)$ \eqref{eq:S} and of its partner $(T)$ \eqref{eq:T}, defined by
$(\alpha_1,\beta_1)$. Let us write in particular
$(\alpha_2^{(1)},\beta_2^{(1)})$ for the transform  of the couple
$(\alpha_2,\beta_2)$. We then apply the associated Darboux transformations to
$(S_1)$ giving rise to $(S_2)$. The couple $(\alpha_3,\beta_3)$ is transformed
into a solution $(\alpha_3^{(2)},\beta_3^{(2)})$ of partner $(T_2)$. Etc\dots
Although we speak of transformed systems to keep track of the coupling, each
of the associated Schr\"odinger equation $- f'' + V^\pm f = k^2 f$ is
transformed by $f \mapsto f' - \frac{g'}g f$ where $g$ is a solution of $ -g''
+ V^\pm g = -\kappa^2 g$, hence independently of what happens to the other
equation $-\phi'' + V^\mp \phi = k^2\phi$. One part of the  Theorem of Crum \cite{Crum}
(which we do not reprove here) tells us that if we apply $n$ successive
Darboux transformations  $f \mapsto f' - \frac{g'}g f$, first by $g_1$, then
by the transformed $g_2$, then by the transformed $g_3$, \dots the final
action can be written directly as:
\begin{equation}
  \label{eq:crum}
  f \mapsto \frac{W(g_1, \dots, g_n, f)}{W(g_1, \dots, g_n)}
\end{equation}
Hence definitions \eqref{eq:crum1} and \eqref{eq:crum2} of $a_n$ and $b_n$ can
be viewed as the final result of the $n$ successive simultaneous Darboux
transformations. Theorem
\ref{thm:simdarboux} tells us how $\mu$ changes when system $(S)$ is
transformed once, hence iterative use of the Theorem gives a formula for
$\mu_n$ involving in fact telescopic products of quotients of Wronskians,
hence equation \eqref{eq:muN}. Moreover if a tau function is initially chosen
with $-(\log\tau)'' = \mu^2$, Theorem \ref{thm:simdarboux} can again be
applied iteratively, leading to a function $\tau_n$ given by \eqref{eq:tauN},
and verifying $-(\log\tau_n)'' = \mu_n^2$. 
\end{proof}

Let us take note that $\mu_n^2 + \mu_n' = - (\log\tau)'' - (\log W(\alpha_1,
\dots,\alpha_n))'' - (\log W(\beta_1, \dots,
      \beta_n))'' + \mu' - \frac{d^2}{dx^2}\log \frac{W(\alpha_1, \dots,
      \alpha_n)}{W(\beta_1, \dots,
      \beta_n)} = \mu^2 + \mu' - 2 \frac{d^2}{dx^2}\log W(\alpha_1,
\dots,\alpha_n)$. And similarly $\mu_n^2 - \mu_n' =  \mu^2- \mu' - 2 \frac{d^2}{dx^2}\log W(\beta_1,
\dots,\beta_n)$. Thus: 

\begin{corollary}
Using the notations of Theorem \ref{thm:crumsystem}, there holds
\begin{subequations}
  \begin{align}
  -a_n'' + V_n^+ a_n  &= k^2 a_n\\
-b_n'' + V_n^- b_n &= k^2 b_n
  \end{align}
\end{subequations}
with
\begin{subequations}
  \begin{align}
  V_n^+ &= V^+ - 2\frac{d^2}{dx^2} \log W(\alpha_1, \dots,
      \alpha_n)\\
  V_n^- &= V^- - 2\frac{d^2}{dx^2} \log W(\beta_1, \dots,
      \beta_n)
  \end{align}
\end{subequations}
\end{corollary} 
These formulas are the part of Crum's Theorem \cite{Crum}
regarding the effect of successive Darboux transformations on the potentials
of Schr\"odinger equations. 

\section{Modification of mu-functions by trivial zeros}

We are interested in Hilbert spaces $H_x$ of entire functions in the sense of
\cite{Bra}, whose reproducing kernels are given  by formula \eqref{eq:AB},
where the functions $A$ ($=A_x$) and $B$ ($=B_x$) are real valued on the real line, 
respectively even and odd, and obey a first order differential system with respect
to $x$ of the Krein type \cite{KreinScattering}, with a real valued
coefficient function $\mu(x)$:
\begin{subequations}
  \begin{align}\label{eq:diffABa}
  \frac{d}{dx} A_x(w) - \mu(x) A_x(w) &=  - w B_x(w)\\ \label{eq:diffABb}
\frac{d}{dx} B_x(w) + \mu(x) B_x(w) &= w A_x(w)
  \end{align}
\end{subequations}
\begin{remark} More general integral equations play the important general
structural role in \cite{Bra}.  We have found that the above restricted type
arises naturally in the study of some specific instances of Hilbert spaces of
entire functions \cite{cras2003}. It turns out to be  well adapted to the
present study of the classical Paley-Wiener spaces modified by adding trivial
zeros on the imaginary axis. If we remove the restriction for the zeros to lie
on the imaginary axis, the functions $A_x$ and $B_x$ real on the real line
will (generally speaking) cease to be respectively even and odd and they obey
the more general type of Krein system from \cite{KreinContinual} which has
both the real and imaginary parts of a complex valued $\mu$-function as
coefficients.
\end{remark}

We want to combine Theorem \ref{thm:3} and Theorem \ref{thm:crumsystem}. We
will suppose that the functions $A_x$ are even, the functions $B_x$ odd, and
the trivial zeros $z_i$, $1\leq i \leq n$, are purely imaginary and verify
$z_i \neq \pm z_j$ for all $i,j$.

From \eqref{eq:diffABa} and \eqref{eq:diffABb}: 
\begin{subequations}
  \begin{align}\label{eq:yy}
  &\left[(\frac{d}{dx} + \mu)
  (\frac{d}{dx} - \mu)\right]^{2p} A_x = (-1)^p w^{2p} A_x \\
\label{eq:zz}
(\frac{d}{dx} - \mu)&\left[(\frac{d}{dx} + \mu)
  (\frac{d}{dx} - \mu)\right]^{2p} A_x = (-1)^{p+1} w^{2p+1} B_x
  \end{align}
\end{subequations} By recurrence the left side of \eqref{eq:yy} (resp. \eqref{eq:zz})
is $(\frac{d}{dx})^{(2p)} A_x$ (resp. $(\frac{d}{dx})^{(2p+1)} A_x$) up  to a
finite linear combination of lower derivatives of $A_x$ with coefficients
being function of $x$ (independent of $w$). Hence, for $n = 2m$:
\begin{equation}
 W(A_x(z_1), \dots, A_x(z_{n}), A_x(w)) =   \begin{vmatrix}
      A_x(z_1) & \dots & A_x(z_n) &A_x(w)\\
      z_1 B_x(z_1) & \dots & z_n B_x(z_n) & w B_x(w)\\
      z_1^2 A_x(z_1) & \dots & z_n^2 A_x(z_n) & w^2 A_x(w)\\
      \vdots & \dots & \dots& \vdots\\
z_1^{2m} A_x(z_1) & \dots & \dots & w^{2m} A_x(w)
    \end{vmatrix}_{(2m+1)\times(2m+1)} 
\end{equation}
and for $n = 2m+1$: 
\begin{equation}
 W(A_x(z_1), \dots, A_x(z_{n}), A_x(w)) =  (-1)^{m+1}  \begin{vmatrix}
      A_x(z_1) & \dots & A_x(z_n) &A_x(w)\\
      z_1 B_x(z_1) & \dots & z_n B_x(z_n) & w B_x(w)\\
      z_1^2 A_x(z_1) & \dots & z_n^2 A_x(z_n) & w^2 A_x(w)\\
      \vdots & \dots & \vdots& \vdots\\
z_1^{2m} A_x(z_1) & \dots & \dots & w^{2m} A_x(w)\\
z_1^{2m+1} B_x(z_1) & \dots & \dots & w^{2m+1} B_x(w)
    \end{vmatrix}_{(2m+2)\times(2m+2)} 
\end{equation} If we divide  the Wronskians (constructed with derivations with
respect to $x$) and the $(n+1)\times(n+1)$
determinants at the right by their respective $n\times n$ principal minors,
the resulting fractions will thus coincide up to  $(-1)^m$ for $n=2m$ and
$(-1)^{m+1}$ for $n = 2m+1$, hence in both cases up to $(-1)^{\frac12 n(n+1)}$. We
can thus rewrite the function $\cA_\sigma$ of Theorem \ref{thm:3} as:
\begin{equation}
  \cA_\sigma(w) = (-1)^n \gamma(w)\gamma^*(w) \frac{W(A_x(z_1), A_x(z_2), \dots,
    A_x(z_n), A_x(w))}{W(A_x(z_1), A_x(z_2), \dots, A_x(z_n))}
\end{equation}

In the same manner
\begin{subequations}
  \begin{align}\label{eq:yyb}
  &\left[(\frac{d}{dx} - \mu)
  (\frac{d}{dx} + \mu)\right]^{2p} B_x = (-1)^p w^{2p} B_x \\
\label{eq:zzb}
(\frac{d}{dx} + \mu)&\left[(\frac{d}{dx} - \mu)
  (\frac{d}{dx} + \mu)\right]^{2p} B_x = (-1)^{p} w^{2p+1} A_x
  \end{align}
\end{subequations} By recurrence the left side of \eqref{eq:yyb} (resp. \eqref{eq:zzb})
is $(\frac{d}{dx})^{(2p)} B_x$ (resp. $(\frac{d}{dx})^{(2p+1)} B_x$) up  to a
finite linear combination of lower derivatives of $B_x$ with coefficients
being function of $x$ (independent of $w$). Hence, for $n = 2m$:
\begin{equation}
 W(B_x(z_1), \dots, B_x(z_{n}), B_x(w)) =   (-1)^m\begin{vmatrix}
      B_x(z_1) & \dots & B_x(z_n) &B_x(w)\\
      z_1 A_x(z_1) & \dots & z_n A_x(z_n) & w A_x(w)\\
      z_1^2 B_x(z_1) & \dots & z_n^2 B_x(z_n) & w^2 B_x(w)\\
      \vdots & \dots & \dots& \vdots\\
z_1^{2m} B_x(z_1) & \dots & \dots & w^{2m} B_x(w)
    \end{vmatrix}_{(2m+1)\times(2m+1)} 
\end{equation}
and for $n = 2m+1$: 
\begin{equation}
 W(B_x(z_1), \dots, B_x(z_{n}), B_x(w)) =  \begin{vmatrix}
      B_x(z_1) & \dots & B_x(z_n) &B_x(w)\\
      z_1 A_x(z_1) & \dots & z_n A_x(z_n) & w A_x(w)\\
      z_1^2 B_x(z_1) & \dots & z_n^2 B_x(z_n) & w^2 B_x(w)\\
      \vdots & \dots & \vdots& \vdots\\
z_1^{2m} B_x(z_1) & \dots & \dots & w^{2m} B_x(w)\\
z_1^{2m+1} A_x(z_1) & \dots & \dots & w^{2m+1} A_x(w)
    \end{vmatrix}_{(2m+2)\times(2m+2)} 
\end{equation}
If we divide  the Wronskians and the determinants at the right by their
respective $n\times n$ principal minors, the results will coincide up to 
$(-1)^m$ for $n=2m$ and $(-1)^{m}$ for $n = 2m+1$, hence in both cases up to
$(-1)^{\frac12 n(n-1)}$. We can  rewrite the function $\cB_\sigma$ of
Theorem \ref{thm:3} as:
\begin{equation}
  \cB_\sigma(w) = (-1)^n \gamma(w)\gamma^*(w) \frac{W(B_x(z_1), B_x(z_2), \dots,
    B_x(z_n), B_x(w))}{W(B_x(z_1), B_x(z_2), \dots, B_x(z_n))}
\end{equation}
where the Wronskians are constructed with derivations with respect to $x$.

Taking into account that $(-1)^n \gamma^*(w) = \gamma(-w)$  we can thus sum up
these computations in the following:

\begin{theorem}\label{thm:wronsk} Let there be given Hilbert spaces  $H_x$ of entire functions,
with functions $A_x$ (even, real on the real line)  and $B_x$ (odd, real on
the real line) computing the evaluators in $H_x$ by formula \eqref{eq:AB}, and
whose variations with respect to the parameter $x$ are given by:
\begin{subequations}
  \begin{align}
  \frac{d}{dx} A_x(w) - \mu(x) A_x(w) &=  - w B_x(w)\\
\frac{d}{dx} B_x(w) + \mu(x) B_x(w) &= w A_x(w)
  \end{align}
\end{subequations}
Let $\sigma = (z_1, \dots, z_n)$ be a finite sequence of purely imaginary
numbers (the associated evaluators in $H_x$ being supposed linearly
independent) with $z_i\neq \pm z_j$ for  $1\leq i, j \leq n$ and
let $H_x(\sigma)$ be the Hilbert space $H_x$ modified by $\sigma$.
Its evaluators $\cK_z$ are given by:
  \begin{equation}
    \cK_z(w) = (\cK_w,\cK_z) = 2 \frac{\overline{\cB_\sigma(z)}\cA_\sigma(w) - \overline{\cA_\sigma(z)}\cB_\sigma(w)}{\overline z -w}
  \end{equation}
with
\begin{subequations}
  \begin{align}
  \cA_{x,\sigma}(w) &= \gamma(w)\gamma(-w) \frac{W(A_x(z_1), A_x(z_2), \dots,
    A_x(z_n), A_x(w))}{W(A_x(z_1), A_x(z_2), \dots, A_x(z_n))}\\ 
  \cB_{x,\sigma}(w) &= \gamma(w)\gamma(-w) \frac{W(B_x(z_1), B_x(z_2), \dots,
    B_x(z_n), B_x(w))}{W(B_x(z_1), B_x(z_2), \dots, B_x(z_n))}
  \end{align}
\end{subequations}
where the Wronskians involve derivatives with respect to the variable $x$. The entire functions $\cA_{x,\sigma}$ and $\cB_{x,\sigma}$ are real on the
real line, and respectively even and odd.
\end{theorem}

Taking into account Theorem \ref{thm:crumsystem} we thus learn that:

\begin{theorem}\label{thm:xxx} Let there be given Hilbert spaces  $H_x$ of entire functions,
functions $A_x$ and $B_x$, imaginary numbers $z_1, z_2, \dots$ verifying the
hypotheses of  Theorem \ref{thm:wronsk}. Let  $H_x(n) = H_x(z_1, \dots, z_n)$
and let  the 
functions $\cA_{x,n}$ and $\cB_{x,n}$ computing the reproducing kernel in
  $H_x(n)$ be provided by Theorem
\ref{thm:3}. They are obtained by successive transformations (essentially) of Darboux type:
  \begin{subequations}
  \begin{align}\label{eq:dna}
    \cA_{x,n+1}(w) &= \frac1{z_{n+1}^2 - w^2} \left(\frac{d}{dx}\cA_{x,n}(w) -
      \frac{\frac{d}{dx}\cA_{x,n}(z_{n+1})}{\cA_{x,n}(z_{n+1})}\cA_{x,n}(w)\right)\\ 
\label{eq:dnb}    \cB_{x,n+1}(w) &= \frac1{z_{n+1}^2 - w^2} \left(\frac{d}{dx}\cB_{x,n}(w) -
      \frac{\frac{d}{dx}\cB_{x,n}(z_{n+1})}{\cB_{x,n}(z_{n+1})}\cB_{x,n}(w)\right)
    \end{align}
\end{subequations} 
and verify the equations
  \begin{subequations}
  \begin{align}\label{eq:muna}
    \frac{d}{dx}\cA_{x,n}(w) - \mu_{n}(x)\; \cA_{x,n}(w) &= - w \cB_{x,n}(w)\\
\label{eq:munb}    \frac{d}{dx}\cB_{x,n}(w) + \mu_{n}(x)\; \cB_{x,n}(w) &= + w
    \cA_{x,n}(w)
    \end{align}
\end{subequations}
  with 
  \begin{equation}
    \mu_n = \mu - \frac{d}{dx}\log \frac{W(A_x(z_1), \dots,
      A_x(z_n))}{W(B_x(z_1), \dots,
      B_x(z_n))}
  \end{equation}
If a
function $\tau$ is chosen with $\mu^2 = - \frac{d^2}{dx^2}\log \tau$ then
$\mu_n^2 = - \frac{d^2}{dx^2}\log \tau_n$ with
  \begin{equation}
    \tau_n = \tau\cdot W(A_x(z_1), \dots,
    A_x(z_n)) W(-iB_x(z_1), \dots,
    -iB_x(z_n))
  \end{equation}
\end{theorem}


In the following, the index $x$ shall be dropped from the notations.
Combining \eqref{eq:dna} with \eqref{eq:muna} we obtain:
\begin{equation}
  \begin{split}
    (z_{n+1}^2 - w^2)\cA_{n+1}(w) = \mu_n\cA_n(w) - w\cB_n(w) - (\mu_n -
    z_{n+1}\frac{\cB_n(z_{n+1})}{\cA_n(z_{n+1})})\cA_n(w) \\= -(z_{n+1}+w)\cB_n(w)
    +\frac{z_{n+1}}{\cA_n(z_{n+1})}\left( \cB_n(z_{n+1})\cA_n(w)+
      \cA_n(z_{n+1})\cB_n(w)\right)\\ = -(z_{n+1}+w)\cB_n(w)
    +\frac{z_{n+1}}{2\cA_n(z_{n+1})} (z_{n+1} + w) K^n({z_{n+1}},w)
  \end{split}
\end{equation}
We have written $K^n(z,w) = K^n_z(w)$ for the evaluator in $H(n) = H(z_1, \dots,
z_n)$. 
Combining \eqref{eq:dnb} with \eqref{eq:munb} gives:
\begin{equation}
  \begin{split}
    (z_{n+1}^2 - w^2)\cB_{n+1}(w) = -\mu_n\cB_n(w) + w\cA_n(w) - (-\mu_n +
    z_{n+1}\frac{\cA_n(z_{n+1})}{\cB_n(z_{n+1})})\cB_n(w) \\= (z_{n+1}+w)\cA_n(w)
    -\frac{z_{n+1}}{\cB_n(z_{n+1})}\left( \cA_n(z_{n+1})\cB_n(w)+
      \cB_n(z_{n+1})\cA_n(w)\right)\\ = (z_{n+1}+w)\cA_n(w)
    -\frac{z_{n+1}}{2\cB_n(z_{n+1})} (z_{n+1} + w) K^n({z_{n+1}},w)
  \end{split}
\end{equation}
We thus have the identities:
\begin{theorem}\label{thm:ntonplus1}
Let $H = H_x$, $\cA_n$, $\cB_n$, for $n\geq 1$  be as in Theorem
\ref{thm:xxx}. There holds 
  \begin{subequations}
  \begin{align}
    (w- z_{n+1})\cA_{n+1}(w) &= \cB_n(w) - \frac{z_{n+1}}{2\cA_n(z_{n+1})}
    K^n(z_{n+1},w) \\ (w - z_{n+1})\cB_{n+1}(w) &= -\cA_n(w) +
    \frac{z_{n+1}}{2\cB_n(z_{n+1})} K^n(z_{n+1},w)
    \end{align}
\end{subequations}
where $K^n(z,w) = K^n_z(w)$ is the reproducing kernel in $H(n) = H(z_1, \dots,
z_n)$. 
\end{theorem}

From formula \eqref{eq:K} in Theorem \ref{thm:gram} we know that $\prod_{1\leq
  i \leq n}(w-z_i)\cdot K^n(z_{n+1},w)$ is a linear combination of the initial evaluators
$Z_i(w)$ ($=Z_{z_i}(w)$), $1\leq i \leq n+1$. Hence by induction we obtain the following:
\begin{theorem}\label{thm:EtoEn}
  Let $H = H_x$, $\cA_n$, $\cB_n$, for $n\geq 1$  be as in Theorem
\ref{thm:xxx}. Let $\cE_n = \cA_n - i \cB_n$ and $\cF_n = \cE_n^* = \cA_n + i
\cB_n$. The function $(-i)^n \prod_{1\leq i\leq n}(w-z_i)\cdot \cE_n(w)$ differs from the initial $E = A - i
  B$ function by a finite linear combination of the initial  evaluators
$Z_i(w)$, $1\leq i \leq n$. Also the function $i^n \prod_{1\leq i\leq
  n}(w-z_i)\cdot \cF_n(w)$ ($\cF_n = \cE_n^*$) differs from the initial $F = A + i
  B$ function by a finite linear combination of the initial  evaluators
$Z_i(w)$, $1\leq i \leq n$.
\end{theorem}

\begin{remark} Let us note that this characterizes uniquely the $\cE_n$ (and
$\cF_n$) provided by theorem \ref{thm:wronsk}, as the unknown linear
combinations of (linearly independent) evaluators will be constrained by
their values at the $z_i$'s.   This theorem for the transition from $H$ to
$H(\sigma)$ holds with much greater
generality than achieved here (see the companion article \cite{trivial}): it suffices for $H$ to verify the axioms of
\cite{Bra}.   Thus, reverting the steps we could have
started from the results proven in \cite{trivial} and, under the additional
hypotheses made here (existence of a parameter $x$ and of a differential system
of Krein type, imaginary trivial zeros,
\dots), obtain the Darboux transformations (\eqref{eq:dna} and
\eqref{eq:dnb} in Theorem \ref{thm:xxx}) and later the Wronskian formulas
(Theorem \ref{thm:wronsk}) as corollaries. 
\end{remark}

\section{Non-linear equations for
  Paley-Wiener spaces with trivial zeros}

On the basis of Theorem \ref{thm:EtoEn} it is convenient to work with the
``incomplete'' forms of the various objects encountered. As the main results
of this chapter are for the classical Paley-Wiener spaces $PW_x$, we will from
the start assume $H = PW_x$. We
consider its  modification $H(\sigma)$ by finitely many ``trivial'' distinct zeros
$\sigma = (z_1, \dots,z_n)$ (the associated evaluators in $H$ are always 
linearly independent). Let $\gamma(w) = \prod_{1\leq j\leq n}
\frac1{w-z_j}$ be the corresponding gamma factor. We define the incomplete
version $K^\sigma(z,w)$ of the reproducing kernel $\cK(z,w)$ in $H(\sigma)$ via the relation
\begin{equation}
  \cK(z,w) = \cK_z(w) = (\cK_w, \cK_z)  =  \gamma(w)\overline{\gamma(z)}
  K^\sigma(z,w)
\end{equation}
Theorem \ref{thm:gram} is the
statement that  $K^\sigma(z,w)$ 
is the unique entire function of $w$ which vanishes at $z_1$, \dots,
$z_n$ and differs additively from
the initial evaluator $Z(z,w)$ by  a finite linear combination of the
initial evaluators $Z(z_1,w)$, \dots,
$Z(z_n,w)$.

Let us now consider the functions $\cE_\sigma$ and $\cF_\sigma$ characterized
as in Theorem \ref{thm:EtoEn}. We
consider their incomplete versions, up to a factor $i^n$:
\begin{equation}
\label{eq:Eincomplete} 
\cE_\sigma(w) = i^n \gamma(w) E_\sigma(w) \qquad 
\cF_\sigma(w) = i^n \gamma(w) F_\sigma(w)
\end{equation}
Of course, there does not hold (for $n\geq1$) $F_\sigma = E_\sigma^*$ (this last function has its trivial zeros not at the $z_i$'s but at the
$\overline{z_i}$'s). The formula for the incomplete reproducing kernel is
\begin{equation}\label{eq:Kincomplete}
K^\sigma(z,w) = \frac{\overline{E_\sigma(z)}E_\sigma(w) - \overline{F_\sigma(z)}F_\sigma(w)}{i(\overline z-w)}  
\end{equation} The rationale for the $i^n$  in \eqref{eq:Eincomplete} is
twofold: first Theorem \ref{thm:EtoEn}, second the fact that if the $z_i$'s
are imaginary the function
$\cA_\sigma$ and $i\cB_\sigma$ obtained in Theorem \ref{thm:3}  are real on
$i\RR$, hence $\cE_\sigma$ and $\cF_\sigma$ are real on $i\RR$, hence
$E_\sigma(it)$ and $F_\sigma(it)$ as defined by  \eqref{eq:Eincomplete}  are
real for $t$ real. The differential
system with respect to $x$ for $E_\sigma$ and
$F_\sigma$  (as
for their complete versions $\cE_\sigma$, $\cF_\sigma$) is:
  \begin{subequations}
  \begin{align}\label{eq:diffEFa}
    \frac{d}{dx} E_\sigma(it) &= t E_\sigma(it) + \mu_\sigma(x) F_\sigma(it)\\
\label{eq:diffEFb}  \frac{d}{dx} F_\sigma(it) &= -t F_\sigma(it) + \mu_\sigma(x) E_\sigma(it)
    \end{align}
\end{subequations}
We introduce the coefficients $c_1$, \dots, $c_n$, $d_1$, \dots, $d_n$
which are the functions of $x$ and of the imaginary points $z_1 = -i \kappa_1$, \dots, $z_n =
- i \kappa_n$ such that, according to Theorem \ref{thm:EtoEn}, the following holds:
\begin{subequations}
  \begin{align}\label{eq:lesc}
  E_\sigma(it) &= e^{xt} + \sum_{1\leq j \leq n} c_j \frac{2\sh((t - 
    \kappa_j)x)}{t - \kappa_j}\\ \label{eq:lesd}
 F_\sigma(it) &= (-1)^n e^{-xt} + \sum_{1\leq j \leq n} d_j \frac{2\sh((t -
   \kappa_j)x)}{t - \kappa_j}
  \end{align}
\end{subequations}
The identity following from $\cF_\sigma = \cE_\sigma^*$ is (we
use that $E_\sigma$ and $F_\sigma$ are real valued on
$i\RR$):
\begin{equation}\label{eq:EvsF}
  \prod_j (t-\kappa_j) F_\sigma(it) = (-1)^n \prod_j (t+\kappa_j) E_\sigma(-it)
\end{equation}
If one is interested in explicit formulas for the $c_j$'s and $d_j$'s, the
initial recipe is to put $t = -\kappa_1$, \dots, $t= -\kappa_n$ in \eqref{eq:lesc}
(resp. \eqref{eq:lesd}) and to use the trivial zeros $E_\sigma(-i\kappa_j) =
0$ (resp. $F_\sigma(-i\kappa_j) =
0$). Cramer's formulas thus lead to determinantal representations for the
$c_j$'s and $d_j$'s (which are seen to be real valued).

\begin{remark} We  pause here to explain how to
remove the restrictions $\kappa_i + \kappa_j \neq 0$. These constraints go
back to Theorem \ref{thm:3}. They were necessary to avoid vanishing of the
denominators $U_n$ and $V_n$, in the formulas for $\cA_n$, $\cB_n$. But
\eqref{eq:lesc} and \eqref{eq:lesd} define $E_\sigma$ and $F_\sigma$, and the
validity of \begin{equation} K^\sigma(z,w) =
\frac{\overline{E_\sigma(z)}E_\sigma(w) -
\overline{F_\sigma(z)}F_\sigma(w)}{i(\overline z-w)}  
\end{equation}   follows by continuity (for real $\kappa_i$'s), as there is no
singularity arising in the formulas for the coefficients $c_1$, \dots,
$c_n$, $d_1$, \dots, $d_n$. The same remark applies to the mu-function
$\mu_\sigma$ which will be expressed below in terms of these
coefficients. Hence  by continuity we again have a mu-function and a
differential system \eqref{eq:diffEFa}, \eqref{eq:diffEFb}  even when
$\kappa_i + \kappa_j = 0$ for some $(i,j)$. The conditions $\kappa_i +
\kappa_j \neq 0$  were inforced
only in order to facilitate the writing of explicit formulas of Wronskian type
for the $\cA$'s and $\cB$'s.
\end{remark}

There is a
plethora of various algebraic and differential identities involving the $c_j$'s
and $d_j$'s. We propose a basic selection, sufficient for our goal in
this chapter. 
From \eqref{eq:lesc}, the value of $(\frac{d}{dx} - t)E_\sigma(it)$ is
\begin{equation} \sum_{1\leq j \leq n} (c_j'
- \kappa_j c_j)\frac{2\sh((t -  \kappa_j)x)}{t - \kappa_j} + \sum_{1\leq j
\leq n} c_j (2\ch((t -\kappa_j)x) - 2\sh((t - \kappa_j)x))
\end{equation}
Comparison with \eqref{eq:diffEFa} gives:
\begin{equation}\label{eq:muformula}
  \mu_\sigma(x) = (-1)^n \sum_{1\leq j
\leq n} 2 c_j e^{\kappa_j x}
\end{equation}
\begin{equation}\label{eq:dcj}
 \text{and}\qquad 1\leq j\leq n\implies\quad \frac{d}{dx} c_j - \kappa_j c_j = \mu_\sigma d_j
\end{equation}
Similarly, from \eqref{eq:lesd}, the value of $(\frac{d}{dx} + t)F_\sigma(it)$
is
\begin{equation} \sum_{1\leq j \leq n} (d_j'
+ \kappa_j d_j)\frac{2\sh((t -  \kappa_j)x)}{t - \kappa_j} + \sum_{1\leq j
\leq n} d_j (2\ch((t -\kappa_j)x) + 2\sh((t - \kappa_j)x))
\end{equation}
Thus:
\begin{equation}
  \mu_\sigma(x) = \sum_{1\leq j \leq n} 2 d_j e^{-\kappa_j x}
\end{equation}
\begin{equation}\label{eq:ddj}
  \text{and}\qquad  1\leq j\leq n\implies\quad \frac{d}{dx} d_j + \kappa_j d_j =  \mu_\sigma c_j
\end{equation}
We take note of the asymptotics:
\begin{subequations}
  \begin{align}
\label{eq:asym1}
  E_\sigma(it) &=_{t\to+\infty} e^{xt} \left( 1 -
    \frac{\alpha_\sigma(x)}{2t} + O(t^{-2})\right)&\qquad \alpha_\sigma(x) &= - 2\sum_{1\leq j \leq
      n} c_je^{-\kappa_j x}\\
\label{eq:asym4}
 F_\sigma(it) &=_{t\to-\infty}  (-1)^n e^{-xt}\left(1 + \frac{\delta_\sigma(x)}{2t} +
   O(t^{-2})\right)&\qquad \delta_\sigma(x) &=  -(-1)^n 2\sum_{1\leq j \leq
      n} d_je^{\kappa_j x}
  \end{align}
\end{subequations}
Using \eqref{eq:EvsF} we obtain $\delta_\sigma(x) = \alpha_\sigma(x) + 4
\sum_{1\leq j\leq n} \kappa_j$. Further,
\begin{equation}
  \frac{d}{dx} \alpha_\sigma(x) = -2 \sum_{1\leq j \leq
      n} (c_j' - \kappa_j c_j) e^{-\kappa_j x} = -2\mu_\sigma\sum_{1\leq j \leq
      n} d_j  e^{-\kappa_j x} = - \mu_\sigma^2
\end{equation}
Using either the differential equations or the identities already known
provides the two further asymptotics:
\begin{subequations}
  \begin{align}\label{eq:asym2}
 F_\sigma(it) &=_{t\to+\infty} e^{xt} \left(\frac{\mu_\sigma(x)}{2t} +
   O(t^{-2})\right)\\
\label{eq:asym3}
  E_\sigma(it) &=_{t\to-\infty} (-1)^n e^{-xt} \left( 
    \frac{-\mu_\sigma(x)}{2t} + O(t^{-2})\right)
  \end{align}
\end{subequations}

We definitely switch to viewing functions depending on $x$ as functions of
the variable $a = e^{-x}$. For example we write $\mu_\sigma(a)$, rather than
$\mu_\sigma(-\log(a))$. We have $a\frac{d}{da} = - \frac{d}{dx}$.  We also fix
once and for all an integer $n\geq1$, and will study the spaces $PW_x(\nu,n)$
associated to a sequence of trivial zeros $z_1$, \dots, $z_n$, in arithmetic
progression:
\begin{equation}
  \kappa_1 =\frac{\nu+1}2, \kappa_2 = \frac{\nu+1}2 + 1, \dots, \kappa_n = \frac{\nu+1}2 +
  n-1\qquad z_j = - i \kappa_j
\end{equation}
The transition from $n$ to $n+1$ is described by Theorem \ref{thm:ntonplus1}. Here
$n$ is fixed, and we shall study the relation between $\nu$ and $\nu+1$. 

We will use the notations $E_\nu$, $E_{\nu+1}$, $F_\nu$, $F_{\nu+1}$,
$\mu_\nu$, $\mu_{\nu+1}$, and $K^\nu$, $K^{\nu+1}$ for the incomplete
reproducing kernel. Neither the dependency on $a$ nor on $n$ is   explicitely recalled in the
notation. Also we shall write $c_1^\nu$, \dots, $c^\nu_n$ and $d^\nu_1$, \dots, $d^\nu_n$,
respectively 
$c^{\nu+1}_1$, \dots, $c^{\nu+1}_n$,  and $d^{\nu+1}_1$, \dots, $d^{\nu+1}_n$, for the coefficients appearing
in equations \eqref{eq:lesc} and \eqref{eq:lesd} for $\nu$ and $\nu+1$. These coefficients are functions of $a$ (depending
on $n$). We rewrite \eqref{eq:lesc} and \eqref{eq:lesd} as 
\begin{subequations}
  \begin{align}\label{eq:lesc2}
  E_\nu(it) &= e^{xt} + 
  \int_{-x}^x e^{-ty}e^{\frac{\nu+1}2 y}\sum_{1\leq j \leq n} c^{\nu}_j\, e^{(j-1)
    y}\,dy\\ 
\label{eq:lesd2}
  F_\nu(it) &= (-1)^ne^{-xt} +  
  \int_{-x}^x e^{-ty}e^{\frac{\nu+1}2 y}\sum_{1\leq j \leq n} d^\nu_j\, e^{(j-1)
    y}\,dy
  \end{align}
\end{subequations}
According to \eqref{eq:lesc2}:
\begin{equation}
  a^{\frac12} E_\nu(i(t+\frac12)) = e^{xt} +  
  \int_{-x}^x e^{-ty}e^{\frac{\nu+2}2 y} \sum_{0\leq j \leq n-1}
  a^{\frac12} c_{j+1}^\nu e^{(j-1)
    y} \,dy
\end{equation}
So the function $w\mapsto a^{\frac12} E_\nu(w+i\frac12) - E_{\nu+1}(w)$ is a
finite linear combination of the $n+1$ initial Paley-Wiener evaluators
$Z(-i\frac{\nu}2, w)$, $Z(-i\frac{\nu+2}2, w)$, \dots, $Z(-i\frac{\nu+2}2 -
i (n-1), w)$. Moreover it has trivial zeros at the trivial zeros of $E_{\nu+1}$.
By Theorem \ref{thm:gram} this identifies $a^{\frac12} E_\nu(w+i\frac12) -
E_{\nu+1}(w)$ with a constant multiple of $K^{\nu+1}(-i\frac{\nu}2, w)$, the factor
being precisely $a^{\frac12} c_{1}^\nu$. We prove in this manner the first of
the following identities:
\begin{proposition}\label{eq:recurrence}
  There holds
  \begin{subequations}
  \begin{align}
\label{eq:rec1}     a^{\frac12} E_\nu(w+i\frac12) &=  E_{\nu+1}(w) + a^{\frac12} c_{1}^\nu
        K^{\nu+1}(-i\frac{\nu}2,w)\\
\label{eq:rec2}         a^{-\frac12} F_\nu(w+i\frac12) &=  F_{\nu+1}(w) + a^{-\frac12} d_{1}^\nu
     K^{\nu+1}(-i\frac{\nu}2,w)\\
\label{eq:rec3}         a^{-\frac12} E_{\nu+1}(w-i\frac12) &=  E_{\nu}(w) + a^{-\frac12} c_{n}^{\nu+1}
     K^{\nu}(-i\frac{\nu+1}2 - in,w)\\
\label{eq:rec4}         a^{\frac12} F_{\nu+1}(w-i\frac12) &=  F_{\nu}(w) +
a^{\frac12}  d_{n}^{\nu+1}
     K^{\nu}(-i\frac{\nu+1}2 - in,w)
    \end{align}
\end{subequations}
\end{proposition}
\begin{proof}
  The additional three identities are proved by the same method as the
  first. We will not make direct use of the ensuing relation between the two
  kinds of evaluators.
\end{proof}

Let us recall that:
\begin{subequations}
  \begin{align}\label{eq:ggg}
  K^{\nu+1}(-i\frac{\nu}2,it) &= \frac{E_{\nu+1}(-i\frac{\nu}2)E_{\nu+1}(it) -
    F_{\nu+1}(-i\frac{\nu}2)F_{\nu+1}(it)}{t - \frac{\nu}2}\\
\label{eq:ggg2}
  K^{\nu}(-i(\frac{\nu+1}2 + n),it) &= \frac{E_{\nu}(-i(\frac{\nu+1}2+n))E_{\nu}(it) -
    F_{\nu}(-i(\frac{\nu+1}2+n))F_{\nu}(it)}{t - \frac{\nu+1}2 - n}
  \end{align}
\end{subequations}
In order to shorten the formulas we adopt the notations:
\begin{subequations}
  \begin{align}
  e_{\nu} &= E_{\nu}(-i\frac{\nu+1}2-in)&  g_{\nu+1}&= E_{\nu+1}(-i\frac\nu2)\\
 f_\nu &= F_{\nu}(-i\frac{\nu+1}2-in)& h_{\nu+1}&= F_{\nu+1}(-i\frac\nu2)
  \end{align}
\end{subequations}
Combining \eqref{eq:rec1} and \eqref{eq:ggg} gives  (with the notation $\sh_x(y) = \frac{\sh(xy)}y$)
\begin{equation}
  \begin{split}
    ({t - \frac{\nu}2})\left(\sum_{0\leq j \leq n-1} a^{\frac12} c_{j+1}^{\nu}
      2\sh_x({t - \frac{\nu}2 - j}) - \sum_{1\leq j \leq n} c_j^{\nu+1}
      2\sh_x({t - \frac{\nu}2 - j})\right)\\ =  a^{\frac12}c_{1}^\nu ({g_{\nu+1} E_{\nu+1}(it) -
      h_{\nu+1} F_{\nu+1}(it)})
  \end{split}
\end{equation}
Using $(t - \frac\nu2)\sh_x(t - \frac\nu2 - j) = \sh(x(t - \frac\nu2 - j)) +j
\sh_x(t - \frac\nu2 - j)$ we can rearrange and obtain identities by termwise
identifications. We record only the identity corresponding to the term with the highest
value of $j$ ($j = n$):
\begin{equation}
  n\, c_n^{\nu+1} = a^{\frac12} c_{1}^\nu \left( h_{\nu+1} d_n^{\nu+1} - g_{\nu+1} c_n^{\nu+1} \right)
\end{equation}
Combining similarly \eqref{eq:rec2} with \eqref{eq:ggg}
we obtain (among other identities!)
\begin{equation}
  n d^{\nu+1}_n = a^{-\frac12} d^\nu_1\left(h_{\nu+1} d^{\nu+1}_n
    - g_{\nu+1}  c^{\nu+1}_n\right)
\end{equation}
Dealing in the same manner with \eqref{eq:rec3}n \eqref{eq:rec4} and
\eqref{eq:ggg2} gives two further
identities of a symmetric type. Summing up, the four obtained relations are:
\begin{proposition}\label{prop:nlcd}
There holds:
\begin{subequations}
  \begin{align}\label{eq:rr1}
  n\, c_n^{\nu+1} &= a^{\frac12} c_{1}^\nu \left( h_{\nu+1} d_n^{\nu+1} - g_{\nu+1} c_n^{\nu+1} \right)
\\
\label{eq:rr2}
  n d^{\nu+1}_n &= a^{-\frac12} d^\nu_1\left(h_{\nu+1} d^{\nu+1}_n
    - g_{\nu+1}  c^{\nu+1}_n\right)
\\
\label{eq:rr3}
 {- n c_1^\nu} &= a^{-\frac12}c^{\nu+1}_n \left({{f_\nu d^\nu_1}
    - e_\nu {c^\nu_1  }}\right)
\\
\label{eq:rr4}
 {-n d^\nu_1} &= a^{\frac12}d^{\nu+1}_n \left({{f_\nu d^\nu_1}
    - e_\nu {c^\nu_1  }}\right)
  \end{align}
\end{subequations}
\end{proposition}
We will need the following:
\begin{proposition}
  None of  the eight quantities $e_\nu$, $f_\nu$, $g_{\nu+1}$, $h_{\nu+1}$,
  $c_1^\nu$, $d^\nu_1$, $c_{n+1}^{\nu+1}$, and $d_{n+1}^{\nu+1}$ can vanish.
\end{proposition}
\begin{proof}
The proof introduces tacitly a link with techniques of orthogonal polynomial
theory, which leads to some results we will expose elsewhere. Let us write explicitely the system of linear equations for the
$c_\nu^j$'s:
\begin{equation}
  \forall i = 1\dots n\quad \sum_{1\leq j \leq n} c_j^\nu \int_{-x}^x e^{(\kappa_i
    + \kappa_j)y}\,dy  = -e^{-x\kappa_i}
\end{equation}
where we recall that $\kappa_j = \frac{\nu+1}2 + j -1$. Let $G$ be the matrix
of this system, we have:
\begin{equation}
  \det(G) c_1^\nu =
  \begin{vmatrix}
      -e^{-x\kappa_1} & \int_{-x}^x e^{(\kappa_1
    + \kappa_2)y}\,dy & \dots & \int_{-x}^x e^{(\kappa_1
    + \kappa_n)y}\,dy  \\
      -e^{-x\kappa_2} & \int_{-x}^x e^{(\kappa_2
    + \kappa_2)y}\,dy & \dots & \int_{-x}^x e^{(\kappa_2
    + \kappa_n)y}\,dy  \\
\dots & \dots & \dots \\
     -e^{-x\kappa_n} & \int_{-x}^x e^{(\kappa_n
    + \kappa_2)y}\,dy & \dots & \int_{-x}^x e^{(\kappa_n
    + \kappa_n)y}\,dy  
 \end{vmatrix}
\end{equation}
We now exploit the relation $\kappa_{j+1} = \kappa_j + 1$ to transform by row
manipulations the determinant on the right into:
\begin{multline}
  \begin{vmatrix}
      -e^{-x\kappa_1} & \int_{-x}^x e^{(\kappa_1
    + \kappa_2)y}\,dy & \dots & \int_{-x}^x e^{(\kappa_1
    + \kappa_n)y}\,dy  \\
      0 & \int_{-x}^x e^{(\kappa_2
    + \kappa_2)y}(1 - e^{-x}e^{-y})\,dy & \dots & \int_{-x}^x e^{(\kappa_2
    + \kappa_n)y}(1 - e^{-x}e^{-y})\,dy  \\
0 & \dots & \dots & \dots \\
     0  & \int_{-x}^x e^{(\kappa_n
    + \kappa_2)y}(1 - e^{-x}e^{-y})\,dy & \dots & \int_{-x}^x e^{(\kappa_n
    + \kappa_n)y}(1 - e^{-x}e^{-y})\,dy  
 \end{vmatrix}\\
=
  -e^{-x\kappa_1} \begin{vmatrix}
      \int_{-x}^x e^{(\kappa_2
    + \kappa_2)y}(e^y - e^{-x})e^{-y}\,dy & \dots & \int_{-x}^x e^{(\kappa_2
    + \kappa_n)y}(e^y - e^{-x})e^{-y}\,dy  \\
\dots &\dots & \dots \\
     \int_{-x}^x e^{(\kappa_n
    + \kappa_2)y}(e^y - e^{-x})e^{-y}\,dy & \dots & \int_{-x}^x e^{(\kappa_n
    + \kappa_n)y}(e^y - e^{-x})e^{-y}\,dy  
 \end{vmatrix}
\end{multline}
This new determinant is a Gramian for a positive measure, it is strictly
positive. So $c_1^\nu < 0$. It can be proven in the
exact same manner  $d_1^\nu >0$, $(-1)^n c_n^{\nu+1} >0$, $(-1)^n d_n^{\nu+1}
<0$, $(-1)^n e_\nu >0$, $(-1)^n f_\nu>0$, $g_{\nu+1} >0$, and $h_{\nu+1}>0$.
\end{proof}

Using the asymptotics for $t\to-\infty$ \eqref{eq:asym4} and \eqref{eq:asym3}
in \eqref{eq:rec1} (and \eqref{eq:ggg}) and also in \eqref{eq:rec3} (and
\eqref{eq:ggg2}), and also the asymptotics \eqref{eq:asym1}, \eqref{eq:asym2}
($t\to+\infty$)  in \eqref{eq:rec2} and \eqref{eq:rec4} give the following
proposition.
\begin{proposition}\label{prop:murec}
  There holds
\begin{subequations}
  \begin{align}\label{eq:murec1}
a\mu_\nu - \mu_{\nu+1} &=  2 a^{\frac12}\; c_1^\nu h_{\nu+1} = -2a^{\frac12}\;
c^{\nu+1}_n f_{\nu}\\ \label{eq:murec2}
a\mu_{\nu+1} - \mu_{\nu} &=  - 2 a^{\frac12} \; d^\nu_1
  g_{\nu+1} =  2 a^{\frac12} \; d^{\nu+1}_n
  e_{\nu}
  \end{align}
\end{subequations}
\end{proposition}
\begin{proposition}
  There holds:
  \begin{subequations}
  \begin{align}\label{eq:comp1}
    c^\nu_1 h_{\nu+1}  &= - c^{\nu+1}_n f_{\nu} \\
\label{eq:comp2}    d^\nu_1 g_{\nu+1}  &= - d^{\nu+1}_n e_{\nu}\\
\label{eq:comp3}   a\, c_1^\nu g_{\nu+1} &= - c^{\nu+1}_n e_{\nu} \\
\label{eq:comp4}    d^\nu_1 h_{\nu+1} &= - a\, d^{\nu+1}_n f_{\nu} 
    \end{align}
\end{subequations}
\end{proposition}
\begin{proof}
  The first two follow from the previous proposition. The last two follow from
  the first and the relation $\frac{a c_1^\nu}{d_1^\nu} =
  \frac{c_n^{\nu+1}}{d_n^{\nu+1}}$ (from \eqref{eq:rr1},
  \eqref{eq:rr2}). Alternatively they can be
  deduced from a look at the asymptotics in
  \eqref{eq:rec1}, \eqref{eq:rec3} for $w=it$, $t\to+\infty$ and in
  \eqref{eq:rec2} and \eqref{eq:rec4} for $t\to-\infty$.
\end{proof}

  \begin{definition}\label{def}
    We define the quantities $X = X(\nu,n,a)$ and $Y(\nu,n,a)$ by the following
    expressions:
    \begin{subequations}
  \begin{align}
      X &:= \frac{g_{\nu+1}}{h_{\nu+1}} = \frac1a \frac{e_\nu}{f_\nu}\\
      Y &:= - \frac1a \frac{c^{\nu+1}_n}{d^{\nu+1}_n} = - \frac{c_1^\nu}{d^\nu_1}
      \end{align}
\end{subequations}
  \end{definition}

From Proposition \ref{prop:murec} we have 
$\mu_\nu = \frac{2a^{\frac12}}{1 - a^2}\left(-ah_{\nu+1} c_1^\nu  + g_{\nu+1} d_1^\nu
  \right)$ and using
the first two relations in Proposition \ref{prop:nlcd} gives
\begin{equation}
  \mu_\nu = \frac{2n a}{1 - a^2} \frac{-h_{\nu+1} c_n^{\nu+1} + g_{\nu+1}
    d_n^{\nu+1}}{h_{\nu+1} d_n^{\nu+1} - g_{\nu+1} c_n^{\nu+1}} = \frac{2n
    a}{1 - a^2} \frac{X+aY}{1+aXY}
\end{equation}
Similarly from \eqref{eq:murec1} and \eqref{eq:murec2}
$\mu_{\nu+1}  = \frac{2a^{\frac12}}{1-a^2}\left(- h_{\nu+1} c_1^\nu  + a
  g_{\nu+1} d_1^\nu \right)$ which gives using Proposition \ref{prop:nlcd} and
Definition \ref{def}
\begin{equation}
  \mu_{\nu+1} = \frac{2n
    a}{1 - a^2} \frac{aX+Y}{1+aXY}
\end{equation}

\begin{theorem}\label{thm:nonlinear}
  The quantities $X$ and $Y$ from Definition \ref{def} are related to the
  mu-functions $\mu_\nu$ and $\mu_{\nu+1}$  by the equations:
  \begin{equation}
    \mu_\nu = \frac{2n a}{1 - a^2}\frac{X+aY}{1+a XY} \qquad \mu_{\nu+1} = \frac{2n a}{1 - a^2}\frac{aX+Y}{1+a XY}
  \end{equation}
They obey the following non-linear differential system:
\begin{subequations}
  \begin{align}\label{eq:nlxyx}
  a\frac{d}{da} X &= \nu
  X - (1 - X^2)\frac{2n a}{1 - a^2}\frac{aX+Y}{1+a XY} \\ \label{eq:nlxyy}
a\frac{d}{da} Y &= -(\nu+1) Y
 + (1 - Y^2)\frac{2n a}{1 - a^2}\frac{X+aY}{1+a XY} 
  \end{align}
\end{subequations}
\end{theorem}
\begin{proof}
    From $X =
    \frac{g_{\nu+1}}{h_{\nu+1}}$ and \eqref{eq:diffEFa}, \eqref{eq:diffEFb}:
      \begin{subequations}
  \begin{align}
     -a\frac{d}{da} g_{\nu+1} &= -\frac\nu2 g_{\nu+1} + \mu_{\nu+1}
     h_{\nu+1} \\
     -a\frac{d}{da} h_{\nu+1} &= \frac\nu2 h_{\nu+1} + \mu_{\nu+1}
     g_{\nu+1}
        \end{align}
\end{subequations}
Hence: $a\frac{d}{da} X = \nu  X - \mu_{\nu+1} (1 - X^2)$. And
from $Y = -\frac{c_1^\nu}{d^\nu_1}$ and
\eqref{eq:dcj}, \eqref{eq:ddj} we have:
\begin{subequations}
  \begin{align}
  -a\frac{d}{da} c_1^\nu&=  \frac{\nu+1}2 c_1^\nu + \mu_\nu d_1^\nu\\
  -a\frac{d}{da} d_1^\nu &= -\frac{\nu+1}2 d_1^\nu +  \mu_\nu c_1^\nu
  \end{align}
\end{subequations}
and this gives $a\frac{d}{da} Y = -(\nu+1) Y + \mu_{\nu} (1 - Y^2)$.
\end{proof}

\begin{theorem}\label{thm:pVIsys}
Let $(X,Y)$ be two functions of a variable $a$. If they obey  the differential system  $(VI_{\nu,n})$:
  \begin{subequations}
  \begin{align} 
  a\frac{d}{da} X &= \nu
  X - (1 - X^2)\frac{2n a}{1 - a^2}\frac{aX+Y}{1+a XY} \\  
a\frac{d}{da} Y &= -(\nu+1) Y
 + (1 - Y^2)\frac{2n a}{1 - a^2}\frac{X+aY}{1+a XY} 
  \end{align}
\end{subequations}
then the quantity $q = a \frac{aX+Y}{X+aY}$
satisfies as function of $b= a^2$  the $P_{VI}$ differential equation:
\begin{equation}\label{eq:PVI}
\begin{split}
  \frac{d^2 q}{db^2} = \frac12\left\{\frac1q+\frac1{q-1}
  +\frac1{q-b}\right\}\left(\frac{dq}{db}\right)^2 - \left\{\frac1b+\frac1{b-1}
  +\frac1{q-b}\right\}\frac{dq}{db}\\+\frac{q(q-1)(q-b)}{b^2(b-1)^2}\left\{\alpha+\frac{\beta
    b}{q^2} + \frac{\gamma (b-1)}{(q-1)^2}+\frac{\delta
    b(b-1)}{(q-b)^2}\right\}
\end{split}
\end{equation}
with
parameters
$(\alpha,\beta,\gamma,\delta) = (\frac{(\nu+n)^2}2,
\frac{-(\nu+n+1)^2}2, \frac{n^2}2, \frac{1 - n^2}2)$.
\end{theorem}
\begin{proof} The computation being lengthy we only give some brief
indications. It is useful to introduce the variable
  \begin{equation}
    T := \frac1{1 + a XY}
  \end{equation}
It verifies the differential equation:
\begin{equation}
  a\frac{d}{da} T = \frac{2 a^2 n}{1 - a^2} T^2 ( Y^2 - X^2)
\end{equation}
Let $b = a^2$. One has
\begin{equation}\label{eq:diffT}
  \frac{d}{db} T = T(T-1)\frac{n(q^2 -b)}{b(q-b)(q-1)}
\end{equation}
On the other hand one establishes:
\begin{equation}\label{eq:diffq}
  \frac{dq}{db} = \frac{2n + \nu + (\nu+1)b}{b(b-1)} q -
  \frac{(\nu+n)q^2+(\nu+1+n)b}{b(b-1)}+ \frac{2nqT}b
\end{equation} This equation gives an expression of $T$ in terms of
$\frac{dq}{db}$, $q$, and $b$. Substituting this value of $T$ in
\eqref{eq:diffT} gives a second order differential equation for $q$. With
some tenacity one finally realizes that it is nothing else than $P_{VI}$ with
parameters $(\alpha,\beta,\gamma,\delta) = (\frac{(\nu+n)^2}2,
\frac{-(\nu+n+1)^2}2, \frac{n^2}2, \frac{1 - n^2}2)$.
\end{proof}

\begin{remark}
  One also establishes that as functions of $b = a^2$, the expressions
  $\frac{-aY}X$ and $a^2 X^2$ are  $P_{VI}$-transcendents.
\end{remark}

We mention finally  the following:
\begin{theorem}
  Let $(X,Y)$  be a solution of the non-linear system $(VI_{\nu,n})$. Let $Z$
  be defined by the following  relation ($n\neq 1$):
  \begin{equation}
    \frac{aY+Z}{1+aYZ} = ({\frac1a-a})\frac{\nu+1}{n-1}\,\frac{Y}{1-Y^2} - \frac{n}{(n-1)}\frac{aY+X}{1+aXY}
  \end{equation}
Then $(Y,Z)$ is a solution of system $(VI_{\nu+1,n-1})$.
Let $W$ be such that
\begin{equation}
  \frac{W + aX}{1 + aWX} = ({\frac1a-a}) \frac\nu{n+1}\,\frac{ X}{ 1 -
    X^2} - 
  \frac{n}{(n+1)} \frac{aX + Y}{1+ aXY}
\end{equation}
Then $(W,X)$ is a solution of system $(VI_{\nu-1,n+1})$.
\end{theorem}

\begin{proof} One needs to eliminate $X$ (resp. $Y$) from the result of
computing $\frac{dZ}{da}$ (resp. $\frac{dW}{da}$). This is a long computation.
\end{proof}

\section{Conclusion}

We assemble some of our main results in the following summary:

\begin{theorem} Let $x>0$, $a = e^{-x}$, and let $PW_x$  be the  Paley-Wiener
space of entire functions which are square integrable on the real line and of
finite exponential type at most $x>0$. Let
 \begin{equation}
   \sigma = (z_1 = -i \kappa_1, \dots, z_n = -i \kappa_n)
 \end{equation}
be a finite sequence of $n$ distinct purely imaginary numbers and 
\begin{equation}
  PW_x(\sigma) = \{ \frac{f(z)}{\prod_{1\leq j\leq n} (z-z_j)}\;| f\in PW_x,
  f(z_1) = \dots = f(z_n) = 0\}
\end{equation} The modified space $PW_x(\sigma)$ is a Hilbert space through
its identification with the closed subspace of functions in $PW_x$ vanishing
on $\sigma$. There is a unique entire function $\cE_\sigma$ verifying the
conditions:
\begin{enumerate}
\item $\cE_\sigma(it)$ is real for $t$ real,
\item $\cE_\sigma(0) >0$,
\item $\lim_{t\to+\infty} \frac{\cE_\sigma(-it)}{\cE_\sigma(it)} = 0$,
\end{enumerate}
and in terms of which the scalar produts of evaluators  in $PW_x(\sigma)$ are:
\begin{equation}\label{eq:Kfinal}
\cK_\sigma(z,w) = \frac{\overline{\cE_\sigma(z)}\cE_\sigma(w) -
  {\cE_\sigma(\overline z)}\overline{\cE_\sigma(\overline w)}}{i(\overline z-w)}  
\end{equation} 
Let $\cF_\sigma(w) = \overline{\cE_\sigma(\overline w)}$ ($ = \cE_\sigma(-w)$). 
There holds:
\begin{equation}
    -a\frac{d}{da} \cE_\sigma(w) = -i w \cE_\sigma(w) + \mu_\sigma(a)
    \cF_\sigma(w)
\end{equation}
The function $\mu_\sigma(a)$ admits (among others) the two representations:
\begin{equation}\label{eq:muWchWsh}
  \mu_\sigma(a) = a\frac{d}{da} \log\frac{W(\ch(\kappa_1 x),\dots,
    \ch(\kappa_n x))}{W(\sh(\kappa_1 x),\dots,
    \sh(\kappa_n x))}
\end{equation}
where $W(f_1,\dots,f_n)$ is the Wronskian determinant (with respect to $x = \log\frac1a$) and
\begin{equation}\label{eq:mugram}
  \mu_\sigma(a) = \frac{2(-1)^n}{g_n(a)} \;
  \begin{vmatrix}
      \int_{-x}^x e^{(\kappa_1
    + \kappa_1)y}\,dy  & \int_{-x}^x e^{(\kappa_1
    + \kappa_2)y}\,dy & \dots & \int_{-x}^x e^{(\kappa_1
    + \kappa_n)y}\,dy  & e^{-x\kappa_1} \\
    \int_{-x}^x e^{(\kappa_2
    + \kappa_1)y}\,dy  & \int_{-x}^x e^{(\kappa_2
    + \kappa_2)y}\,dy & \dots & \int_{-x}^x e^{(\kappa_2
    + \kappa_n)y}\,dy & e^{-x\kappa_2}  \\
\dots & \dots & \dots &\dots&\dots\\
     \int_{-x}^x e^{(\kappa_n
    + \kappa_1)y}\,dy & \int_{-x}^x e^{(\kappa_n
    + \kappa_2)y}\,dy & \dots & \int_{-x}^x e^{(\kappa_n
    + \kappa_n)y}\,dy  & e^{-x\kappa_n}\\
e^{x\kappa_1}& \dots &\dots & e^{x\kappa_n} &0
 \end{vmatrix}
\end{equation}
where $g_n(a) = \det(\int_{-x}^x e^{(\kappa_i
    + \kappa_j)y}\,dy)_{1\leq i,j\leq n}$ is the principal $n\times n$ minor
  of the $(n+1)\times(n+1)$
determinant. There also holds
\begin{equation}\label{eq:mu2}
  \mu_\sigma(a)^2 = - a\frac{d}{da}a\frac{d}{da} \log g_n(a)
\end{equation}
In the specific case where $\sigma$ is an arithmetic progression:
\begin{equation}
  \kappa_1 = \frac{\nu+1}2, \kappa_2 = \frac{\nu+1}2 + 1, \dots, \kappa_n =
  \frac{\nu+1}2 + n-1
\end{equation}
the function $\mu_{\nu,n}(a)$ can be expressed as a quotient of two multiple
integrals:
\begin{equation}\label{eq:multint}
  \mu_{\nu,n}(a) = 2n\;\frac{\idotsint_{[a,\frac1a]^{n-1}}  \prod_{i<j}
    (t_j - t_i)^2 \prod_i t_i^\nu(t_i - a)(\frac1a - t_i) dt_1\dots dt_{n-1}}{\idotsint_{[a,\frac1a]^n} \prod_{i<j}
    (t_j - t_i)^2  \prod_i t_i^\nu  dt_1\dots dt_{n}}
\end{equation}
and the expression
\begin{equation}
  q_{\nu,n} = \frac{a \mu_{\nu+1,n}}{\mu_{\nu,n}}\;,
\end{equation}
as a function of $a^2$, verifies
the Painlev\'e~VI equation with
  parameters
  \begin{equation}
(\alpha,\beta,\gamma,\delta) = (\frac{(\nu+n)^2}2,
  \frac{-(\nu+n+1)^2}2, \frac{n^2}2, \frac{1 - n^2}2)
\end{equation}
\end{theorem}
\begin{proof} We have established \eqref{eq:Kfinal} with $\cE_\sigma(it) =
\prod_{1\leq i \leq n}\frac1{t+\kappa_i} E_\sigma(it)$ and $E_\sigma(it)$
given by equation \eqref{eq:lesc}. In particular $E_\sigma$ hence $\cE_\sigma$
is real on the imaginary axis. One has $E_\sigma(it)\sim_{t\to+\infty}
e^{xt}$, hence $\cE_\sigma(it)$, which by \eqref{eq:Kfinal} cannot vanish for
$t>0$, is positive for $t>0$. It is easy to prove that given arbitrary points
$w_1$, \dots, $w_m$ and non-negative integers $n_1$, \dots $n_m$ there is in
$PW_x$ a function vanishing exactly to the order $n_j$ at $w_j$ for all
$j$. So in $PW_x(\sigma)$ the evaluator at $z=0$ is non zero, which proves
$\cE_\sigma(0)\neq0$, hence $\cE_\sigma(0)>0$. As $\cE_\sigma$ is real on the
imaginary line one has $\cF_\sigma(w) = \cE_\sigma(-w)$. And from the
representations \eqref{eq:lesc} and \eqref{eq:lesd} we know
$\frac{\cF_\sigma(it)}{\cE_\sigma(it)}\to_{t\to+\infty} 0$. Let $\cE$ be
another function computing the reproducing kernel and with the  properties
(1), (2) and (3). Let $\cA = \frac12(\cE + \cE^*)$, $\cB = \frac i2(\cE -
\cE^*)$, which are respectively even and odd, real on the real line. The
evaluator at the origin (as said above, necessarily non zero) is
$\cK^\sigma(0,w) = 2\cA(0)\frac1w \cB(w)$, hence the function $\cB$ is known
up to a positive real multiple, then the function $\cA$ is known up to (the
inverse of) this multiple. Condition (3) can be written
$\frac{-i\cB(it)}{\cA(it)}\to_{t\to+\infty} 1$, and this finally identifies
$\cA = \cA_\sigma$ and $\cB = \cB_\sigma$.

Formula \eqref{eq:muWchWsh} (in which $\kappa_i+\kappa_j\neq0$ is assumed)
follows from Theorem \ref{thm:xxx}, and formula \eqref{eq:mu2} from Theorem
\ref{thm:xxx} and Proposition \ref{prop}.
According to \eqref{eq:muformula} one has $  \mu_\sigma(a) = 2(-1)^n \sum_{1\leq j
\leq n} c_j^\sigma e^{\kappa_j x}$, the coefficients $c_j^\sigma$ solving:
\begin{equation}
  \forall i = 1\dots n\quad \sum_{1\leq j \leq n} c_j^\sigma \int_{-x}^x e^{(\kappa_i
    + \kappa_j)y}\,dy  = -e^{-x\kappa_i}
\end{equation}
This gives the representation \eqref{eq:mugram}. 

In the case of the arithmetic progression of reason one, row manipulations
replace the $(n+1)\times(n+1)$ determinant by
\begin{equation}
  - e^{-x\kappa_n}\begin{vmatrix}
      \int_{-x}^x e^{(\kappa_1
    + \kappa_1)y}(1 - \frac1{a} e^y)\,dy  &  \dots & \int_{-x}^x e^{(\kappa_1
    + \kappa_n)y}(1 - \frac1{a} e^y )\,dy \\
    \int_{-x}^x e^{(\kappa_2
    + \kappa_1)y}(1 - \frac1{a} e^y )\,dy  &  \dots & \int_{-x}^x e^{(\kappa_2
    + \kappa_n)y}(1 - \frac1{a} e^y )\,dy \\
\dots & \dots & \dots \\
     \int_{-x}^x e^{(\kappa_{n-1}
    + \kappa_1)y}(1 - \frac1{a} e^y )\,dy & \dots & \int_{-x}^x e^{(\kappa_{n-1}
    + \kappa_n)y}(1 - \frac1{a} e^y )\,dy \\
e^{x\kappa_1}& \dots & e^{x\kappa_n} 
 \end{vmatrix}
\end{equation}
Column manipulations lead to:
\begin{equation}
  \begin{split}
    - \begin{vmatrix} \int_{-x}^x e^{(\kappa_1 + \kappa_1)y}(1 - a e^y)(1 -
      \frac1a e^y)\,dy & \dots & \int_{-x}^x e^{(\kappa_1
        + \kappa_{n-1})y}(1 - a e^y )(1 - \frac1a e^y)\,dy \\
      \int_{-x}^x e^{(\kappa_2 + \kappa_1)y}(1 - a e^y )(1 - \frac1a e^y)\,dy
      & \dots & \int_{-x}^x e^{(\kappa_2
        + \kappa_{n-1})y}(1 - a e^y )(1 - \frac1a e^y)\,dy \\
      \dots & \dots & \dots \\
      \int_{-x}^x e^{(\kappa_{n-1} + \kappa_1)y}(1 - a e^y )(1 - \frac1a
      e^y)\,dy & \dots & \int_{-x}^x e^{(\kappa_{n-1} + \kappa_{n-1})y}(1 - a
      e^y )(1 - \frac1a e^y)\,dy
    \end{vmatrix} \\= (-1)^n \det_{1\leq i, j\leq n-1} \left(\int_a^{\frac1a}
    t^{\nu+i+j-2}(\frac1a- t)(t-a)\,dt\right)
  \end{split}
\end{equation}
So we have the representation of $\mu_{\nu,n}$ as a quotient of two gramians
(for $n=1$ the determinant at the numerator is taken to be $1$):
\begin{equation}\label{eq:mugramquot}
  \mu_{\nu,n}(a) = 2 \frac{\det_{1\leq i, j\leq n-1} \left(\int_a^{\frac1a}
    t^{\nu+i+j-2}(\frac1a- t)(t-a)\,dt\right)}{\det_{1\leq i, j\leq n} \left(\int_a^{\frac1a}
    t^{\nu+i+j-2}\,dt\right)}
\end{equation}
This gives formula
\eqref{eq:multint}. 

Finally the Painlev\'e~VI assertion follows from Theorems \ref{thm:nonlinear} and \ref{thm:pVIsys}. 
\end{proof}

\begin{remark}
  The change of variables $t_j = a + (\frac1a - a)u$ transforms the multiple integrals at the
  numerator and denominator of
  expression \eqref{eq:multint} into, respectively:
\begin{equation}
  (\frac1a - a)^{n^2 - 1}a^{(n-1)\nu}{\idotsint_{[0,1]^{n-1}}  \prod_{i<j}
    (u_j - u_i)^2 \prod_i (1 +\alpha u_i)^\nu u_i (1-u_i) du_1\dots du_{n-1}}
\end{equation}
\begin{equation}
\text{and}\quad   (\frac1a - a)^{n^2}a^{n \nu}{\idotsint_{[0,1]^n} \prod_{i<j}
    (u_j - u_i)^2  \prod_i (1 +\alpha u_i)^\nu  du_1\dots du_{n}}
\end{equation} with $\alpha = \frac1{a^2} - 1$. Similar multiple integrals
arise in the work of Forrester and Witte \cite{fwJUE} relating the ``Jacobi
unitary ensemble'' of random matrices to the Okamoto hamiltonian formulation
of the Painlev\'e~VI equation \cite{okamotoVI}. It seems however (cf. the
introduction of \cite{fwRMT_PVI}) that in this context of multiple integrals
one had so far not yet encountered directly Painlev\'e~VI
transcendents \emph{per se}, but rather solutions to Okamato's
``$\sigma$-equation''. 
\end{remark}

\begin{remark}
  We see from \eqref{eq:mugramquot} or \eqref{eq:mugram} that
  $\mu_{\nu,n}(a)$ is the product of $a^{\nu+1}$ with a rational function of
  $a^2$ and $a^{2\nu}$. The quantity $q = a \frac{\mu_{\nu+1,n}}{\mu_{\nu,n}}$
  is a rational
  function  (with rational coefficients) of $b$  and $b^\nu$ ($b = a^2$). In
  particular, for $\nu\in\ZZ$, $q$ is a rational function of $b$.
\end{remark}

\begin{remark} We have left aside most of the developments which have their
origins in \cite{cras2003}, \cite{hankel}  and which put the objects studied
here in another context (which leads in particular to various further
representations for the $\mu$-functions), indeed a context which presided over
their introduction.  We have also left aside a number of other developments
related to techniques of orthogonal polynomial theory, multiple integrals and
non-linear relations. We hope to address these topics in further
publications.
\end{remark}

\bigskip

\begin{small} \textbf{Acknowledgments}. All results included in this
manuscript were obtained between January and late April 2008, while I was
spending a seven months visit at the I.H.E.S. Numerous conversations on
various (mostly other) topics with permanent members, visitors and staff made
this a particularly stimulating and memorable residence. I also benefited
during this period from electronic exchanges with Philippe Biane. I had found
non-linear equations in the framework I had developed to establish the Fourier
transform as a scattering \cite{cras2003, hankel} and Philippe Biane's
question about what this would give when the Gamma function was replaced by a
finite rational expression gave a welcome impetus to my efforts of that time.
I discovered in the process of answering this question the non-linear
equations \eqref{eq:nlxyx}, \eqref{eq:nlxyy}, from which the previous ones
derived by a process of confluence, and I related them to Painlev\'e~VI. I
thus thank him for many interesting interactions (for example, in relation
with the Wronskians built with the trigonometric functions, which he had also
encountered in his own enterprises),  and for having then kept me informed of
some of his works \cite{biane}. I also thank the organizers Joaquim Burna,
H{\aa}kan Hedenmalm, Kristian Seip, and Mikhail Sodin of the mini-symposium on
Hilbert spaces of entire functions (EMS congress, Amsterdam 2008) for having
given me the opportunity to present there aspects of this material. Finally, I
would like to express special thanks to Michel Balazard for having made
possible my participation to the Zeta Functions congresses I, II, and III in
Moscow, in 2006, 2008, and 2010, which gave me also opportunities to present
related material.
 

\end{small}



\begin{thebibliography}{10}

\setlength{\baselineskip}{10pt} \setlength{\parskip}{0pt}

\bibitem{biane} Ph.~Biane, \emph{Orthogonal polynomials on the unit circle,
q-Gamma weights, and discrete Painlev\'e equations}, arXiv:0901.0947, July
2010 (v1 January 2009), 26 pages.

\bibitem{Bra} L.~de~Branges, \emph{Hilbert spaces of entire
functions}, Prentice Hall Inc., Englewood Cliffs, 1968.


\bibitem{cras2003} {J.-F.~Burnol}, \emph{Des \'equations
de Dirac et de Schr\"odinger pour la transformation de
Fourier},  C. R. Acad. Sci. Paris Ser. I, {336} (2003)
919-924.


\bibitem{hankel} J.-F.~Burnol, \emph{Scattering, determinants,
hyperfunctions in relation to
$\frac{\Gamma(1-s)}{\Gamma(s)}$}, 63
pages, feb 2006. {arXiv:math.NT/0602425}. 

\bibitem{trivial} J.-F.~Burnol, \emph{Hilbert spaces of entire functions with
    trivial zeros}, 10 pages, July 2010. arXiv:

\bibitem{Crum} M. M. Crum, \emph{Associated Sturm-Liouville systems},
  Quart. J. Math. Oxford (2), 6 (1955) 121-127.

\bibitem{Darboux} G. Darboux, \emph{Sur une proposition relative aux \'equations lin\'eaires}, Comptes Rendus Acad. Sci. 94 (1882) 1456-1459.

\bibitem{fwJUE} P.J. Forrester and N.S. Witte, \emph{Application of the  $\tau$-function theory of
Painlev\'e equations to random matrices:
$P_{VI}$, the JUE, CyUE, cJUE and scaled limits}, Nagoya Math. J.
Vol. 174 (2004) 29-114.

\bibitem{fwRMT_PVI} P.J. Forrester and N.S. Witte, \emph{Random matrix theory and the sixth Painlev\'e equation}, J. Phys. A: Math. Gen. 39 (2006) 12211-12233.

\bibitem{KreinScattering} M. G. Krein, \emph{On the determination of the potential of a particle from its
  S-function}, Dokl. Akad. Nauk SSSR 105, no 3. (1955) 433-436.

\bibitem{KreinContinual} M. G. Krein, \emph{Continual analogues of
propositions on polynomials orthogonal on the unit circle}, Dokl. Akad. Nauk SSSR 105, no 4. (1955)
637-640.






\bibitem{Okada} S. Okada, \emph{Applications of Minor Summation Formulas to Rectangular-Shaped Representations of Classical Groups},
J. of Algebra, Vol. 205, no. 2 (1998) 337-367

\bibitem{okamotoVI} K. Okamoto, \emph{Studies on the Painlev\'e equations. I. Sixth Painlev\'e equation P${}_{VI}$}, Ann.
Mat. Pura Appl. (4), 146 (1987) 337-381.


\end{thebibliography}
\end{document}